\documentclass[reqno, 12pt]{amsart}
\usepackage[letterpaper,hmargin=1in,vmargin=1.2in]{geometry}
\usepackage{amsmath, amssymb, amsthm, verbatim, bbm, url}
\usepackage{graphicx}
\usepackage{enumerate}
\usepackage{amsmath,amssymb,amsthm,mathrsfs,graphicx,url}
\usepackage[usenames,dvipsnames]{color}
\usepackage[colorlinks=true,linkcolor=Red,citecolor=Green]{hyperref}

\newcommand \N {\mathbb{N}}
\newcommand \R {\mathbb{R}}
\newcommand \C {\mathbb{C}}
\newcommand \Z {\mathbb{Z}}

\newcommand \Hh {\mathbb{H}}

\newcommand \Oh {\mathcal{O}}

\newcommand \la {\langle}
\newcommand \ra {\rangle}

\newcommand \D {\partial}
\newcommand \eps {\varepsilon}

\newcommand \Def {\stackrel{\textrm{def}}=}
\DeclareMathOperator \re {Re}
\DeclareMathOperator \im {Im}
\DeclareMathOperator \ad {ad}

\DeclareMathOperator \supp {supp}
\DeclareMathOperator \chsupp {chsupp}

\DeclareMathOperator \WF {WF}

\DeclareMathOperator \Op {Op}
\DeclareMathOperator \Id {Id}

\DeclareMathOperator \rank {rank}
\DeclareMathOperator \vol {vol}
\DeclareMathOperator \erfc {erfc}
\newtheorem{prop}{Proposition}
\newtheorem{lem}[prop]{Lemma}
\newtheorem*{thm}{Theorem}
\newtheorem*{conj}{Conjecture}
\theoremstyle{definition}
\numberwithin{equation}{section}
\numberwithin{prop}{section}
\numberwithin{figure}{section}

\title
[Resonance free regions for nontrapping manifolds with cusps]
{Resonance free regions for nontrapping manifolds with cusps}

\author[Kiril Datchev]
{Kiril Datchev}
\address{Mathematics Department, Massachusetts Institute of Technology, Cambridge, MA
02139}
\email{datchev@math.mit.edu}

\begin{document}

\begin{abstract}
We prove resolvent estimates for nontrapping manifolds with cusps which imply the existence of 
arbitrarily wide resonance free strips,  local smoothing for the Schr\"odinger equation, and resonant wave expansions. We obtain lossless limiting absorption and local smoothing estimates, 
but the estimates on the holomorphically continued resolvent exhibit losses. 
We prove that these estimates are optimal in certain respects.
\end{abstract}

\maketitle

\addtocounter{section}{1}
\addcontentsline{toc}{section}{1. Introduction}
\thispagestyle{empty}

Resonance free regions near the essential spectrum have been extensively studied since the foundational work of Lax-Phillips and Vainberg. Their  size  is  related to the dynamical structure of the set of trapped classical trajectories. More trapping typically results in a smaller region, and the largest resonance free regions exist when there is no trapping.

\noindent \textbf{Example.} Let $\Hh^2$ be the hyperbolic upper half plane. Let $(X, g)$ be a nonpositively curved, compactly supported metric perturbation of the quotient space $\langle z \mapsto z+1 \rangle\backslash \Hh^2$. As we show in \S\ref{s:examples}, there are no trapped geodesics (that is, all geodesics are unbounded).

Let $(X,g)$ be as  above or as in \S \ref{s:assumptions}, with dimension $n+1$ and Laplacian $\Delta \ge 0 $. The resolvent $(\Delta - n^2/4 - \sigma^2)^{-1}$ is holomorphic for $\im \sigma > 0$, except at any $\sigma \in i \R$ such that $\sigma^2 + n^2/4$ is an eigenvalue, and has essential spectrum $\{\im \sigma = 0\}$: see Figure \ref{f:intro}. 

\begin{thm}
 For all $\chi \in C_0^\infty(X)$, there exists $M_0 > 0$ such that for all $M_1>0$ there exists $M_2 >0$ such that the cutoff resolvent $\chi (\Delta - n^2/4 - \sigma^2)^{-1} \chi$
continues holomorphically  to $\{|\re \sigma| \ge M_2,\, \im \sigma \ge - M_1\}$, where it obeys the estimate
\begin{equation}\label{logreg}
\|\chi (\Delta - n^2/4 - \sigma^2)^{-1} \chi\|_{L^2(X) \to L^2(X)} \le M_2|\sigma|^{-1 +M_0|\im \sigma|}.
\end{equation}
\end{thm}

\begin{figure}[htbp]
\includegraphics[width=15cm]{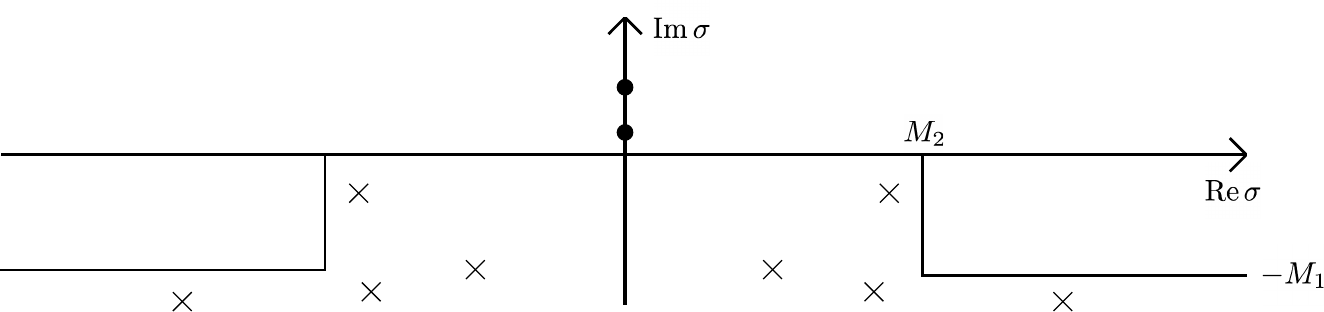}
\caption{
We prove that the cutoff resolvent continues holomorphically to arbitrarily wide strips and obeys polynomial bounds. 
}\label{f:intro}
\end{figure}

In the example above, and in many of the examples in \S\ref{s:examples}, $\chi (\Delta - n^2/4 - \sigma^2)^{-1} \chi$ is meromorphic in $\C$. The poles of the meromorphic continuation  
are called \emph{resonances}.

Logarithmically large resonance free regions go back to work of Regge \cite{regge} on potential scattering. In the setting of obstacle scattering they were found by Lax-Phillips \cite{lp} and Vainberg \cite{v}, and their results were generalized by Morawetz-Ralston-Strauss \cite{mrs} and Melrose-Sj\"ostrand \cite{ms}. When $X$ is Euclidean outside of a compact set, they have been established for very general nontrapping perturbations of the Laplacian by Sj\"ostrand-Zworski in \cite[Theorem 1]{sz}, which extends earlier work of Martinez \cite{m} and Sj\"ostrand \cite{s}. Most recently, Baskin-Wunsch \cite{bw} derive them for geometrically nontrapping manifolds with cone points. These works give a larger resonance free region and a stronger resolvent estimate than the Theorem above, but require asymptotically Euclidean geometry near infinity.

The manifolds considered in this paper are  nontrapping, but the cusp makes them not uniformly so: for a sufficiently large compact set $K \subset X$, we have 
\[\sup_{\gamma \in \Gamma} \textrm{ diam }\gamma^{-1}(K) = +\infty,\]
where $\Gamma$ is the set of unit speed geodesics in $X$. This is because geodesics may travel arbitrarily far into the cusp before escaping down the funnel; this dynamical peculiarity makes it difficult to separate the analysis in the cusp from the analysis in the funnel and is the reason for the relatively involved resolvent estimate gluing procedure  we use below.

Resonance free strips also exist in some trapping situations, with width determined by dynamical properties of the trapped set.  These go back to work of Ikawa \cite{ik}, with recent progress  by Nonnenmacher-Zworski \cite{nz}, Petkov-Stoyanov \cite{ps}, Alexandrova-Tamura \cite{at}, and Wunsch-Zworski \cite{wz}. 
Resonance free regions and resolvent estimates have  applications to evolution equations, and this is an active area: examples include resonant wave expansions and  wave decay, local smoothing estimates,  Strichartz  estimates, geometric control, and wave damping \cite{bur:smoothing, bz, bh, msv, gn, chr:mrl,  bgh,  Dyatlov:Asymptotic, csvw}; see also \cite{wsur} for a recent survey and more references.  In \S\ref{s:applications} we apply  \eqref{logreg} to local smoothing  and resonant wave expansions.

If  $(X,g)$ is evenly asymptotically hyperbolic (in the sense of Mazzeo-Melrose \cite{m} and Guillarmou \cite{g}) and nontrapping, then for any $M_1>0$ there is $M_2>0$ such that
\begin{equation}\label{e:eucbetter} 
\|\chi(\Delta - n^2/4 - \sigma^2)^{-1}\chi\|_{L^2(X) \to L^2(X)} \le M_2|\sigma|^{-1}, \quad |\re \sigma| \ge M_2,\, \im \sigma \ge - M_1,
\end{equation}
by work of Vasy \cite[(1.1)]{Vasy} (see also the analogous estimate for asymptotically Euclidean spaces in Sj\"ostrand-Zworski \cite[Theorem $1'$]{sz}). The bound \eqref{logreg} is weaker due to the presence of a cusp. Indeed, by studying low angular frequencies (which correspond to geodesics which travel far into the cusp before escaping down the funnel) in Proposition \ref{p:bessel} we show that if $(X,g) = \langle z \mapsto z+1 \rangle\backslash \Hh^2$, then
\begin{equation}\label{e:low}
\|\chi(\Delta - n^2/4 - \sigma^2)^{-1}\chi\|_{L^2(X) \to L^2(X)} \ge e^{-C|\im\sigma|}|\sigma|^{-1+2|\im \sigma|}/C,
\end{equation}
for $\sigma$ in the lower half plane and bounded away from the real and imaginary axes. 

The lower bound \eqref{e:low} gives a sense in which \eqref{logreg} is optimal, but  finding the maximal resonance free region remains an open problem. The only known explicit example of this type is $(X,g) = \langle z \mapsto z+1 \rangle \backslash \Hh^2$, for which Borthwick \cite[\S5.3]{b} expresses the resolvent in terms of Bessel functions and shows  
 there is only one resonance and it is simple (see also Proposition \ref{p:bessel}). On the other hand, Guillop\'e-Zworski \cite{gz} study more general surfaces, and prove that if the $0$-volume is not zero, then there are infinitely many resonances and optimal lower and upper bounds hold on their number in disks. We apply their result to our setting in \S\ref{s:examples}, giving a family of surfaces  with infinitely many resonances to which our Theorem applies, but it is not clear even in this case whether or not  the resonance free region given by the Theorem is optimal. The model resolvent bound \eqref{e:modelbound2b} below suggests that, if $(X,g)$ is a surface of revolution, then the methods of \S\ref{s:modelcusp} and \S\ref{s:funnel}, suitably elaborated, will allow one to replace the region $\{|\re \sigma| \ge M_2,\, \im \sigma \ge - M_1\}$  in the Theorem
by the more natural \ $\{|\re \sigma| \ge M_2,\, \im \sigma \ge - M_1 \log\log|\re\sigma|\}$.

In \cite[Corollary 1.2]{cv}, Cardoso-Vodev, extending work of Burq \cite{bu0, bu}, prove resolvent estimates for very general infinite volume manifolds (including the ones studied here; note that the presence of a funnel implies that the volume is infinite) which imply an exponentially small resonance free region. Our Theorem gives the first large resonance free region for a family of manifolds with cusps. 

 For $\im \sigma=0$, \eqref{logreg} is lossless; that is to say it agrees with the result for general nontrapping operators on asymptotically Euclidean or hyperbolic manifolds (see Cardoso-Popov-Vodev \cite[(1.6)]{cpv} and references therein).  However, if $(X,g)$ is asymptotically Euclidean or hyperbolic in the sense of \cite[\S 4]{Datchev-Vasy:Gluing}, then the gluing methods of that paper show that such a lossless estimate for $\im \sigma=0$ implies \eqref{e:eucbetter} for some $M_1>0$; see \cite{d-exten}. In this sense it is due to the cusp that $\Oh(|\sigma|^{-1})$ bounds hold for $\im \sigma=0$ but not in any strip containing the real axis.

The Theorem also provides a first step in support of the following

\begin{conj}[Fractal Weyl upper bound] Let $\Gamma$ be a geometrically finite discrete group of isometries of $\Hh^{n+1}$ such that $X = \Gamma \backslash \Hh^{n+1}$ is a smooth noncompact manifold. Let $R(X)$ denote the set of eigenvalues and resonances of $X$ included according to multiplicity, let $K \subset T^*X$ be the set of maximally extended, unit speed geodesics which are precompact, and let $m$ be the Hausdorff dimension of $K$. Then for any $C_0>0$ there is $C_1 > 0$ such that
\[\#\{\sigma \in R(X)\colon |\sigma - r| \le C_0\} \le C_1 r^{(m-1)/2}.\]
\end{conj}

This statement is a partial generalization to the case of resonances of the Weyl asymptotic for eigenvalues of a compact manifold; such results go back to work of Sj\"ostrand \cite{s}. If $\Gamma\backslash \Hh^{n+1}$ has funnels but no cusps, this is proved in joint work with Dyatlov \cite{dd} (generalizing earlier results of Zworski \cite{z} and Guillop\'e-Lin-Zworski \cite{glz}); if $X = \Gamma \backslash \Hh^2$ has cusps but no funnels, this follows from work of Selberg \cite{sel}. When $n=1$ the remaining case is  $\Gamma\backslash \Hh^2$ having both cusps and funnels. The methods of the present paper, combined with those of \cite{sz, dd}, provide a possible approach to the conjecture in this case. When $n \ge 2$ cusps can have mixed rank, and in this case even meromorphic continuation of the resolvent was proved only recently by Guillarmou-Mazzeo \cite{gm}.


In \S\ref{s:prelim} we give the general assumptions on $(X,g)$ under which the Theorem holds, and deduce  consequences for the geodesic flow and for the spectrum of the Laplacian. We then give examples of manifolds which satisfy the assumptions, including examples with infinitely many resonances and examples with eigenvalue.

In \S\ref{s:reduce} we use a resolvent gluing method, based on one developed in joint work with Vasy \cite{Datchev-Vasy:Gluing}, to reduce the Theorem to proving resolvent estimates and propagation of singularities results for three model operators. The first  model operator  is semiclassically elliptic outside of a compact set, and we analyze it in \S\ref{s:rk} following \cite{sz} and \cite{Datchev-Vasy:Gluing}.

In \S\ref{s:modelcusp} we study the second model operator, the model in the cusp. We use a separation of variables, a semiclassically singular rescaling, and an elliptic variant of the gluing method of \S\ref{s:reduce} to reduce its study to that of a family of one-dimensional Schr\"odinger operators for which uniform resolvent estimates and propagation of singularities results hold. 
The  rescaling causes losses for the resolvent estimate on the real axis, and we remove these by a non-compact variant of the method of propagation of singlarities through trapped sets developed in joint work with Vasy \cite{Datchev-Vasy:Propagation}. The lower bound \eqref{e:low} shows that these losses cannot  be removed for the continued resolvent; see also Bony-Petkov \cite{bp} for related and more general lower bounds in Euclidean scattering.

In \S\ref{s:funnel} we study the third model operator, the model in the funnel, and we again reduce to a  family of one-dimensional Schr\"odinger operators. To obtain uniform estimates we use a variant of the method of complex scaling of Aguilar-Combes \cite{ac} and Simon \cite{sim}, following the geometric approach of Sj\"ostrand-Zworski \cite{sz2}. The method of complex scaling was first adapted to such families of operators by Zworski \cite{z}, but we use here the approach of \cite{Datchev:Thesis}, which is slightly simpler and is adapted to non-analytic manifolds. The analysis in this section could be replaced by that of \cite{Vasy}, which avoids separating variables; the advantage of our approach is that it gives an estimate in a logarithmically large neighborhood of the real axis. Although we do not exploit this here, as mentioned above this improvement can probably be used to show that a larger resonance free region exists, at least when $(X,g)$ is a surface of revolution.


In \S\ref{s:applications} we apply \eqref{logreg} to local smoothing and resonant wave expansions. For the latter  we need the additional assumption, satisfied in the example above and in many of the examples in \S\ref{s:examples}, that $\chi (\Delta - n^2/4 - \sigma^2)^{-1} \chi$ is meromorphic in $\C$. 
In \S\ref{s:low} we prove \eqref{e:low} using Bessel function asymptotics.

I am indebted especially to Maciej Zworski for his generous guidance, advice, and unflagging encouragement throughout the course of this project. Thanks also to Andr\'as Vasy, Nicolas Burq, John Lott, David Borthwick, Colin Guillarmou,  Hamid Hezari,  Semyon Dyatlov, and Richard Melrose for their interest and for their many very helpful ideas, comments, and suggestions. 
 I am also grateful for the hospitality of the Mathematical Sciences Research Institute and of the Universit\'e Paris 13.
 I was partially supported by the National Science Foundation under grant DMS-0654436 and under a postdoctoral fellowship.

\section{Preliminaries}\label{s:prelim}

Throughout the paper $C>0$ is a large constant which may change from line to line, and estimates are always uniform for $h \in (0,h_0]$, where $h_0>0$  may change from line to line.

\subsection{Assumptions}\label{s:assumptions}
Let $S$ be a compact $n$ dimensional boundaryless manifold, and let
\[X = \R_r \times S.\]
Let $R_g > 0$, and let $g$ be a Riemannian metric on $X$ such that 
\begin{equation}\label{e:metricinfinity}
g|_{\{\pm r > R_g\}} = dr^2 + e^{2(r + \beta(r))}dS_\pm,
\end{equation}
where $dS_+$ and $dS_-$ are metrics on $S$, $R_g>0$ and $\beta \in C^\infty(\R)$. We call the region $\{r < -R_g\}$ the \textit{cusp}, and the region $\{r > R_g\}$ the \textit{funnel}. 

\begin{figure}[htbp]
\includegraphics[width=140mm]{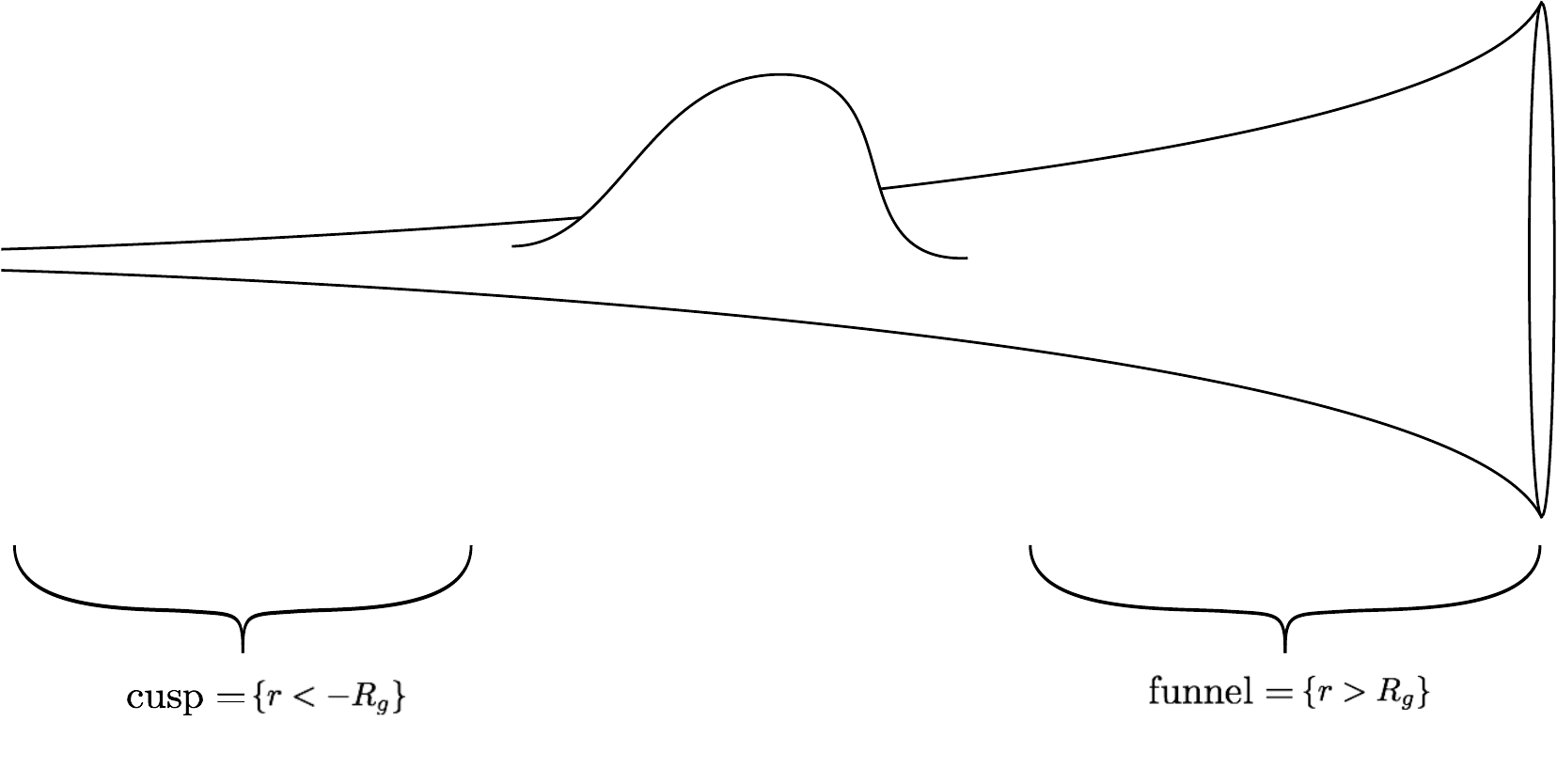}
\caption{The manifold $X$.}\label{f:mfld}
\end{figure}

Suppose there is $\theta_0 \in (0,\pi/4)$ such that $\beta$ is holomorphic and bounded in the sectors $|z| > R_g,\ \min\{|\arg z|,\, |\arg -z|\} < 2\theta_0$. By Cauchy estimates,  for all $k \in \N$ there are $C, C_k >0$,  such that if $|z| > R_g,\ \min\{|\arg z|,\, |\arg -z|\} \le \theta_0$, then 
\[
|\beta^{(k)}(z)| \le C_k |z|^{-k}, \ |\im \beta(z)| \le C |\im z|/|z|.
\]
In particular, after possibly redefining $R_g$ to be larger, we may assume without loss of generality that, for all $r \in \R$,
\begin{equation}\label{e:betahalf}
|\beta'(r)| + |\beta''(r)| \le 1/4.
\end{equation}

In the example at the beginning of the paper $\beta \equiv 0$. When the funnel end is an exact hyperbolic funnel, $\beta(r) =C +  \log(1 + e^{-2r})$ for $r> R_g$.

We make two dynamical assumptions: if $\gamma \colon \R \to X$ is a maximally extended geodesic, assume $\gamma(\R)$ is not bounded and  $\gamma^{-1}(\{r < -R_g\})$ is connected. See \S\ref{s:examples} for examples.


\subsection{Dynamics near infinity}
Let $p+1$ be the geodesic Hamiltonian, that is
\[
p = \rho^2 + e^{-2(r + \beta(r))}\sigma_\pm - 1,
\]
in the region $\{\pm r > R_g\}$, where $\rho$ is dual to $r$, and $\sigma_\pm$ is the geodesic hamiltonian of $(S,dS_\pm)$. From this we conclude that, along geodesic flowlines, we have
\[
\dot r(t) = H_p\rho = 2\rho(t), \qquad
\dot\rho(t) = -H_p r = 2 \left[1 + \beta'(r(t))\right] e^{-2(r + \beta(r))}\sigma_\pm,
\]
so long as the trajectory remains within $\{\pm r > R_g\}$. In particular,
\begin{equation}\label{e:convexity}
\ddot r(t) = 4\left[1 + \beta'(r(t))\right] e^{-2(r + \beta(r))} \sigma_\pm \ge 0.
\end{equation}
Dividing the equation for $\dot \rho$ by $p + 1 - \rho^2$, putting $\hat \rho = \rho/\sqrt{p+1}$, and integrating we find
\begin{equation}\label{tanh}\begin{split}
\tanh^{-1} \hat\rho(t) -  \tanh^{-1}\hat \rho(0) &= 2 \sqrt {p+1} \left(t + \int_0^t\beta'(r(s))ds\right)\\
& \ge \frac 34\ \frac{r(t) - r(0)}{\max\{\hat \rho(s): s \in [0,t]\}} ,
\end{split}\end{equation}
where the equality holds so long as the trajectory remains in $\{\pm r > R_g\}$, and the inequality (which follows from \eqref{e:betahalf} and the equation for $\dot r $) holds when additionally $t \ge 0$, $\rho(0) \ge 0$.

\subsection{The essential spectrum.}\label{spectrum}
The nonnegative Laplacian is given by
\begin{align*} \Delta|_{\{\pm r >R_g\}}  
&= D_r^2 - i n(1 + \beta'(r))D_r + e^{-2(r + \beta(r))}  \Delta_{S_\pm}, \end{align*}
where $D_r = -i\D_r$, and $\Delta_{S_\pm}$ is the Laplacian on $(S,dS_{\pm})$. 
Fix $\varphi\in C^\infty(X)$ such that
\begin{equation}\label{e:phi} \varphi|_{\{|r|>R_g\}} = n(r + \beta(r))/2.\end{equation}
Then 
\begin{equation}\label{e:vdef}
\begin{split}
\left.\left(e^{\varphi} \Delta e^{-\varphi}\right)\right|_{\{\pm r>R_g\}}
&= D_r^2 + e^{-2(r + \beta(r))} \Delta _{S_\pm} + \frac {n^2} 4 + V(r),
\end{split}
\end{equation}
where $V(r) = \varphi'' + {\varphi'}^2  - \frac{n^2}4  =   \frac n 2 \beta'' + \frac{n^2}2 \beta'  + \frac{n^2}4 {\beta'}^2 .$ This shows  the essential spectrum of $\Delta$ is $[n^2/4 ,\infty)$ 
(see for example \cite[Theorem XIII.14, Corollary 3]{rs}); the potential perturbation $V$ is relatively compact since $\beta'$ and $\beta''$ tend to zero at infinity (see for example Rellich's criterion \cite[Theorem XII.65]{rs}).

In this paper we study:
\begin{equation}\label{e:pdef}
P \Def h^2\left(e^{\varphi} \Delta e^{-\varphi} - \frac{n^2}4\right) - 1,
\end{equation}
as an unbounded operator on $L^2_\varphi(X) \Def \{e^\varphi u\colon u \in L^2(X)\}$ with domain \[H^2_\varphi(X) \Def \{u \in L_\varphi^2(X)\colon e^{\varphi} \Delta e^{-\varphi} u \in L_\varphi^2(X)\} = \{e^\varphi u\colon u \in H^2(X)\}.\]
We will show that for every $\chi \in C_0^\infty(X)$,  $E \in (0,1)$ there exists $C_0 > 0$ such that for every $\Gamma >0$ there exist $C,h_0>0$ such that the cutoff resolvent $\chi(P-\lambda)^{-1}\chi$ continues holomorphically from $\{\im \lambda >0\}$ to $[-E,E] - i [0,\Gamma h]$ and satisfies
\begin{equation}\label{e:main}
\|\chi(P - \lambda)^{-1}\chi\|_{L_\varphi^2(X) \to L_\varphi^2(X)} \le C h^{-1-C_0|\im\lambda|/h},
\end{equation}
uniformly for $\lambda \in [-E,E] - i [0,\Gamma h]$ and $h \in (0,h_0]$. This implies the Theorem and \eqref{logreg}.

\subsection{Examples}\label{s:examples}

In this section we give a family of examples of manifolds satisfying the assumptions of \S\ref{s:assumptions}. I am very grateful to John Lott for suggesting this family of examples. In this section $d_g(p,q)$ denotes the distance between $p$ and $q$ with respect to the Riemannian metric $g$, and $L_g(c)$ denotes the length of a curve $c$ with respect to $g$.

Let $(\Hh^{n+1},g_h)$ be hyperbolic space with coordinates
\[
(r,y) \in \R \times \R^n, \qquad g_h = dr^2 + e^{2r} dy^2.
\]
Let $(X, g_h)$ be a parabolic cylinder obtained by quotienting the $y$ variables to a torus:
\[
X = \R \times \left(\la y \mapsto y + c_1, \dots, y \mapsto y + c_n \ra \backslash \R^n\right),
\]
where the $c_j$ are linearly independent vectors in $\R^n $. Let   $R_g> 0$, put $dS_+ = dS_- = dy^2$, and take $\beta \in C^\infty(\R)$ satisfying all assumptions of \S\ref{s:assumptions}, including \eqref{e:betahalf}. On $\{|r| > R_g\}$ define $g$ by \eqref{e:metricinfinity}, and on $\{|r| \le R_g\}$   let $g$ be any metric with all sectional curvatures nonpositive. The calculation in the Appendix shows that the sectional curvatures  in $\{|r| > R_g\}$ are nonpositive so long as \eqref{e:betahalf} holds.

The two dynamical assumptions in the last paragraph of \S\ref{s:assumptions} will follow from the following classical theorem (see for example \cite[Theorem III.H.1.7]{brha}).

\begin{prop}[Stability of quasi-geodesics]
Let $(\Hh^{n+1},g_h)$ be hyperbolic $n+1$-space, let $p,q \in \Hh^{n+1}$, and let $\gamma_h\colon[t_1,t_2] \to \Hh^{n+1}$ be the unit speed geodesic from $p$ to $q$. Suppose $c\colon[t_1,t_2]\to\Hh^{n+1}$ satisfies $c(t_1) = p$, $c(t_2) = q$, and there is $C_1>0$ such that
\begin{equation}\label{quasi}\frac 1 {C_1} |t-t'|\le d_{g_h}(c(t),c(t')) \le C_1 |t-t'|,\end{equation}
for all $t,t'\in[t_1,t_2]$. Then
\begin{equation}\label{conquasi}\max_{t \in [t_1,t_2]}d_{g_h}(\gamma_h(t),c(t)) \le C_2,\end{equation}
where $C_2$ depends only on $C_1$.
\end{prop}

To apply this theorem, observe first that just as $g_h$ descends to a metric on $X$, so $g$ lifts to a metric on $\Hh^{n+1}$; call the lifted metric $g$ as well. Observe there is $C_g$ such that
\begin{equation}\label{methyp}\frac 1{C_g} g_h(u,u) \le g(u,u) \le C_g g_h(u,u), \qquad u \in T_x X, \ x \in X.\end{equation}
Indeed for $x$ varying in a compact set this is true for any pair of metrics, and on $\{|r|>R_g\}$ it suffices if $C_g \ge e^{2\max|\beta|}$. We will show that if $c$ is a unit speed $g$-geodesic in $\Hh^n$, then \eqref{quasi} holds with a constant $C_1$ depending only on $C_g$. Since both $g$ and $g_h$ have nonnegative curvature and hence  distance-minimizing geodesics, it is equivalent to show that
\begin{equation}\label{quasi2}\frac 1 {C_1} d_g(p,q)  \le d_{g_h}(p,q) \le {C_1} d_g(p,q),\end{equation}
holds for all $p,q \in \Hh^{n+1}$, with a constant $C_1$ which depends only on $C_g$. For this last we compute as follows: let $\gamma$ be a unit speed $g$-geodesic from $p$ to $q$. Then
\[\begin{split}d_{g_h}(p,q) \le L_{g_h}(\gamma) = \int_{t_1}^{t_2} \sqrt{g_h(\dot\gamma,\dot\gamma)}dt \le\int_{t_1}^{t_2} \sqrt{C_g g(\dot\gamma,\dot\gamma)}dt 
= \sqrt{C_g}L_g(\gamma)=\sqrt{C_g}d_g(p,q). \end{split}\]
This proves the second inequality of \eqref{quasi2}, and the first follows from the same calculation since  \eqref{methyp} is unchanged if we switch $g$ and $g_h$.

Let $\gamma \colon \R \to X$ be a  $g$-geodesic and $\gamma_h \colon \R \to X$ a $g_h$-geodesic. For any $x \in X$ we have
\[\lim_{t\to\infty}d_{g_h}(\gamma_h(t),x) = \lim_{t\to\infty}d_{g}(\gamma_h(t),x) = \infty,\]
and by \eqref{conquasi} the same holds if $\gamma_h$ is replaced by $\gamma$. In particular $\gamma(\R)$ is not bounded.

We check finally that $\gamma^{-1}(\{r < -R_g\})$ is connected. It suffices to check that if instead $\gamma\colon \R \to \Hh^{n+1}$ is  a $g$-geodesic, then $\gamma^{-1}(\{r < -N\})$ is connected  for $N$  large enough. We then conclude by redefining  $R_g$ to be larger than $N$.

We argue by way of contradiction. From \eqref{e:convexity} we see that $\dot r(t)$ is nondecreasing along $\gamma$ in $\{r < -R_g\}$. Hence, if $\gamma^{-1}(\{r < - N\})$ is to contain at least two intervals for some $N> R_g$, there must exist times $t_1<t_2<t_3$ such that $r(\gamma(t_1)), r(\gamma(t_3)) < -  N$, $r(\gamma(t_2)) = -R_g$. Now the $g_h$-geodesic $\gamma_h\colon[t_1,t_3] \to \Hh^n$ joining $\gamma(t_1)$ to $\gamma(t_3)$ has $r(\gamma_h(t)) < -N$ for all $t \in [t_1,t_3]$. It follows that $d_{g_h}(\gamma_h(t_2),\gamma(t_2)) \ge N -R_g$, and if $N$ is large enough this violates \eqref{conquasi}.

\subsubsection{Examples with infinitely many resonances.}\label{infmany} In this subsection we specialize to the case $n=1$, $\beta(r) = 0$ for $r < -R_g$, $\beta(r) = \beta_0 + \log(1 + e^{-2r})$ for $r > R_g$ and for some $\beta_0 \in \R$. Then the cusp and funnel of $X$ are isometric to the standard cusp and funnel obtained by quotienting $\Hh^2$ by a nonelementary Fuchsian subgroup (see e.g. \cite[\S2.4]{b}).

In particular there is $\ell >0$ such that
\[X = \R_r \times  (\R/\ell\Z)_t, \qquad g|_{\{r > R_g\}} = dr^2 + \cosh^2r dt^2.\]
If $(X_0, g_0) = [0,\infty) \times (\R/\ell\Z), \ g_0 =  dr^2 + \cosh^2r dt^2,$ then 
the $0$-volume of $X$ is
\[0\,\textrm{-}\vol(X) \Def \vol_g(X \cap \{r < R_g\}) - \vol_{g_0}(X_0 \cap  \{r < R_g\}).\]

Let $R_\chi(\sigma)$ denote the meromorphic continuation of $\chi (\Delta - 1/4 - \sigma^2)^{-1} \chi$. In this case, $R_\chi(\sigma)$ is meromorphic in $\C$ (\cite{mm, gz}), and near each pole $\sigma_0$ we have
\[R_\chi(\sigma) = \chi \left(\sum_{j=1}^k \frac {A_j}{(\sigma - \sigma_0)^j} + A(\sigma)\right) \chi,\]
where the $A_j\colon L^2_{\textrm{comp}}(X) \to L^2_{\textrm{loc}}(X)$  are finite rank  and $A(\sigma)$ is holomorphic near $\sigma_0$. The \emph{multiplicity} of a pole, $m(\sigma_0)$ is given by $m(\sigma) \Def \rank\left(\sum_{j=1}^k A_j\right).$
\begin{prop}\cite[Theorem 1.3]{gz}
If $0$-$\vol(X) \ne 0$, then there exists a constant $C$ such that
\[\lambda^2/C \le \sum_{|\sigma|\le \lambda}m(\sigma) \le C\lambda^2, \qquad \lambda > C.\]
\end{prop}
We can ensure that $0$-$\vol(X) \ne 0$ by adding, if necessary, a small compactly supported metric perturbation to $g$. Then, as $\lambda \to \infty$, the meromorphic continuation of $R_\chi$ will have $\sim \lambda^2$ many poles in a disk of radius $\lambda$, but none of them will be in the strips \eqref{logreg}.

\subsubsection{Examples with eigenvalue}\label{exeigensec}
In this subsection we  consider examples of the form
\begin{equation}\label{exampeigen}X = \R \times (\R^n \slash  \Z^n) \qquad g = dr^2+ \exp \left(2r + 2\int_{-\infty}^r b\right)dy^2, \qquad b \in C_0^\infty(\R).\end{equation}
By  the Appendix, $(X,g)$ is  nonpositively curved if $b' + (b + 1)^2 \ge 0$ everywhere, e.g. if $b \ge -1/2$ and $b' \ge -1/4$; then all the assumptions of \S \ref{s:assumptions} hold. We will give a sufficient condition on $b$ such that $X$ has at least one eigenvalue, and also infinitely many resonances.

By the calculation in \S\ref{spectrum}, if $\varphi(r)=- \frac n 2\left( r +  \int_{-\infty}^r b\right)$ for all $r \in \R$,
then
\[e^{-\varphi} \Delta e^{\varphi} = D_r^2 + e^{-2(r + \int^r b)} \Delta_{\R^n/\Z^n} + \frac{n^2} 4 + V(r), \quad V(r) \Def  \frac n 2 b'(r) + \frac {n^2} 4 b(r)^2 + \frac {n^2} 2 b(r).\]
Observe that $V \in C_0^\infty(\R)$, and consequently (see for example \cite[Theorem XIII.110]{rs}) for $D_r^2 + V(r)$ to have a negative eigenvalue it is sufficient to ensure that
\[\int_{-\infty}^\infty  V(r)dr <  0.\]
But in \cite[Theorem 2]{z87} Zworski shows that if $V \not\equiv 0$, the operator $D_r^2 + V(r)$ has infinitely many resonances: indeed the number in a disk of radius $\lambda$ is given by 
\[\frac 2 \pi |\chsupp V| \lambda + o(\lambda),\qquad \lambda \to \infty,\]
where $\chsupp$ denotes the convex hull of the support. This eigenvalue and these resonances correspond to an eigenvalue and resonances for $\Delta$:  one multiplies the eigenfunction and resonant states by $e^{\varphi}$ and regards them as functions on $X$ which depend on  $r$ only.

In summary if $(X,g)$ is given by \eqref{exampeigen}, then the assumptions of \S\ref{s:assumptions} hold if $b \ge -1/2$, $b' \ge -1/4$. It has infinitely many resonances and at least one eigenvalue if $b \not\equiv 0$, $b \le 0$.

\subsection{Pseudodifferential operators}\label{secpseudor}
In this section we review some facts about semiclassical pseudodifferential operators, following \cite{ds} and \cite{ez}.

\subsubsection{Pseudodifferential operators on $\R^n$} For $m \in \R$, $\delta \in [0,1/2)$ let $S_\delta^m(\R^n)$ be the symbol class of functions $a = a_h(x,\xi) \in C^\infty(T^*\R^n)$ satisfying
\begin{equation}\label{symboldef}
\left|\D^\alpha_x \D^\beta_\xi a\right| \le C_{\alpha,\beta} h^{-\delta(|\alpha| + |\beta|)} (1+|\xi|^2)^{(m-|\beta|)/2},
\end{equation}
uniformly in $T^*\R^n$. The \textit{principal symbol} of $a$ is its equivalence class  in $S_\delta^{m}(\R^n) / h S_\delta^{m-1}(\R^n)$. Let $S^m(\R^n) = S^m_0(\R^n)$.

We quantize $a \in S_\delta^m(\R^n)$ to an operator $\Op(a)$ using the formula
\begin{equation}\label{quantdef}(\Op(a) u)(x) = \frac 1 {(2\pi h)^n} \int\!\!\!\int e^{i(x-y)\cdot\xi/h}a\left(h,x,\xi\right)u(y)dyd\xi,\end{equation}
and put $\Psi_\delta^m(\R^n) = \{\Op(a)| a \in S_\delta^m(\R^n)\}$, $\Psi^m(\R^n) = \Psi^m_0(\R^n)$. If $A = \Op(a)$ then $a$ is the \textit{full symbol} of $A$, and the principal symbol of $A$ is the principal symbol of $a$.
If $A \in \Psi_\delta^m(\R^n)$, then for any $s \in \R$ we have $\|A\|_{H^{s+m}_h(\R^n) \to H^s_h(\R^n)} \le C$, where (if $\Delta \ge0$)
\[\|u\|_{H^s_h(\R^n)} = \|(1 + h^2\Delta)^{s/2}u\|_{L^2(\R^n)}.\]
If $A \in \Psi_\delta^m(\R^n)$ and $B \in \Psi_\delta^{m'}(\R^n)$, then $AB \in \Psi_\delta^{m+m'} (\R^n)$ and $[A,B] = AB - BA \in h^{1-2\delta}\Psi_\delta^{m+m'-1} (\R^n)$. If $a, b$ are  the  principal symbols of $A, B$, then the  principal symbol of $h^{2\delta-1}[A,B]$ is $i H_ba$, where $H_b$ is the Hamiltonian vector field of $b$.

If $K \subset T^*\R^n$ has either $K$ or $T^*\R^n \setminus K$ bounded in $\xi$, then $a \in S_\delta^m(\R^n)$ is \emph{elliptic} on~$K$~if
\begin{equation}\label{ellipdef}|a| \ge  (1+|\xi|^2)^{m/2}/C,\end{equation}
uniformly for $(x,\xi) \in K$. We say that $A \in \Psi_\delta^m(\R^n)$ is elliptic on $K$ if its principal symbol is. For such $K$, we say $A$ is \textit{microsupported} in $K$ if the full symbol $a$ of $A$ obeys 
\begin{equation}\label{e:microsuppdef}
|\D_x^\alpha \D^\beta_\xi a| =  C_{\alpha,\beta,N}h^N (1 + |\xi|^2)^{-N}
\end{equation}
uniformly on $T^*\R^n\setminus K$, for any $\alpha, \beta, N$. 
If $A_1$ is microsupported in $K_1$ and $A_2$ is microsupported in $K_2$, then $A_1A_2$ is microsupported in $K_1 \cap K_2$.

If $A \in \Psi^m_\delta(\R^n)$ is elliptic on $K$, then it is invertible there in the following sense: there exists $G \in \Psi^{-m}_\delta(\R^n)$ such that $AG - \Id$ and $GA - \Id$ are both microsupported in $T^*X \setminus K$. Hence if $B \in \Psi_\delta^{m'}(\R^n)$ is microsupported in $K$ and $A$ is elliptic in an $\eps$-neighborhood of $K$ for some $\eps > 0$, then, for any $s,N \in \R$.
\begin{equation}\label{ellipestrn} \|Bu\|_{H^{s+m}_h(\R^n)} \le C \|ABu\|_{H^{s}_h(\R^n)} + \Oh(h^\infty)\|u\|_{H^{-N}_h(\R^n)}.\end{equation}
The \textit{sharp G\aa rding inequality}  says that if the principal symbol of $A \in \Psi_\delta^m(\R^n)$ is nonnegative near $K$ and $B \in \Psi_\delta^{m'}(\R^n)$ is microsupported in $K$, then
\begin{equation}\label{gardingrn}\la A B u, B u \ra_{L^2(\R^n)} \ge -Ch^{1-2\delta} \|B u\|^2_{H^{(m-1)/2}(\R^n)} - \Oh(h^\infty)\|u\|_{H^{-N}_h(\R^n)}.\end{equation}

\subsubsection{Pseudodifferential operators on a manifold}\label{secpseudoman} 

These results extend to the case of a noncompact manifold $X$, provided we require our estimates to be uniform only on compact subsets of $X$. We formulate our estimates for $L^2_\varphi(X)$ and its associated Sobolev spaces, but of course this choice of density is not essential.

Write $S^m_\delta(X)$ for the symbol class of functions  $a \in C^\infty( T^*X)$ satisfying \eqref{symboldef} on coordinate patches (note that this condition is invariant under change of coordinates). The principal symbol of $a$ is its equivalence class in $S_\delta^m(X) / hS_\delta^{m-1}(X)$, and let $S^m(X) = S^m_0(X)$.

Let $h^\infty \Psi^{-\infty}(X)$ be the set of linear operators $R$ such that for any $\chi \in C_0^\infty(X)$, we have  $\|\chi R\|_{H^{-N}_{\varphi,h}(X) \to H^N_{\varphi,h}(X)} + \|R \chi \|_{H^{-N}_{\varphi,h}(X) \to H^N_{\varphi,h}(X)} \le C h^N$ for any $N$,
where  
\begin{equation}\label{e:hphidef}
\|u\|_{H^s_{\varphi,h}(X)} \Def \|(2+P)^{s/2}u\|_{L_\varphi^2(X).}
\end{equation}
We quantize $a \in S_\delta^{m}(X)$ to an operator $\Op(a)$ by using a partition of unity and the formula \eqref{quantdef} in coordinate patches. Let $\Psi_\delta^{m}(X) = \{\Op(a) + R | a \in S_\delta^m(X), R \in h^\infty\Psi^{-\infty}(X)\}$.  The quantization $\Op$ depends on the choices of coordinates and partition of unity, but the class $\Psi_\delta^{m}(X)$ does not. If $A \in \Psi_\delta^{m}(X)$ and $\chi \in C_0^\infty(X)$, then $\chi A$ and $A \chi$ are bounded $H^{s+m}_{\varphi,h}(X) \to H^{s}_{\varphi,h}(X)$.
If $A \in \Psi_\delta^{m}(X)$ and $B \in \Psi_\delta^{m'}(X)$, then $AB \in \Psi_\delta^{m+m'} (X)$ and  $h^{2\delta-1}[A,B] \in \Psi_\delta^{m+m'-1} (X)$. If $a, b$ are  the  principal symbols of $A$ and $B$ (the principal symbol is invariantly defined, although the total symbol is not), then the  principal symbol of $h^{2\delta-1}[A,B]$ is $i H_ba$, where $H_b$ is the Hamiltonian vector field of $b$. 

Let $K \subset T^*X$ have either $K \cap T^*U$ bounded for every bounded $U \subset X$, or $T^*U \setminus K$ bounded for every bounded $U \subset X$. We say $a \in S_\delta^m(X)$ is \emph{elliptic} on $K$ if \eqref{ellipdef} holds
uniformly on $T^*U \cap K$ for every bounded $U \subset X$. We say that $A \in \Psi_\delta^m(X)$ is elliptic on $K$ if its principal symbol is. We say $A$ is \textit{microsupported} in $K$ if a full symbol $a$ of $A$ obeys \eqref{e:microsuppdef}
uniformly on $T^*U \setminus K$ for every bounded $U \subset X$ and for any $\alpha, \beta, N$ (note that if this  holds for one full symbol of $A$, it also does for all the others). 

If $B \in \Psi^{m'}_\delta(X)$ is microsupported in $K$ and $A$ is elliptic in an $\eps$-neighborhood of $K$ for some $\eps > 0$, then, for any $s,N \in \R$ and $\chi \in C_0^\infty(X)$,
\begin{equation}\label{ellipestx} \|B\chi u\|_{H^{s+m}_{\varphi,h}(X)} \le C \|AB \chi u\|_{H^{s}_{\varphi,h}(X)} + \Oh(h^\infty)\|\chi u\|_{H^{-N}_{\varphi,h}(X)}.\end{equation}
The \textit{sharp G\aa rding inequality}  says that if the principal symbol of $A \in \Psi_\delta^m(X)$ is nonnegative near $K$ and $B \in \Psi_\delta^{m'}(X)$ is microsupported in $K$, then for every $\chi \in C_0^\infty(X)$, $N \in \R$,
\begin{equation}\label{gardingx}\la A B \chi u, B \chi u \ra_{L^2_{\varphi}(X)} \ge -Ch^{1-2\delta} \|B \chi u\|^2_{H^{(m-1)/2}_{\varphi,h}(X)} - \Oh(h^\infty)\|\chi u\|_{H^{-N}_{\varphi,h}(X)}.\end{equation}

\subsubsection{Exponentiation of operators}\label{expop} For $q \in C_0^\infty(T^*X)$, $Q$ a quantization of $q$, and $\eps \in[0,C_0 h\log(1/h)]$, we will be interested in operators of the form $e^{\eps Q/h}$.  We write
\[e^{\eps Q/h} = \sum_{j=0}^\infty \frac {(\eps/h)^j}{j!} Q^j,\]
with the sum converging in the $H^s_{\varphi,h}(X) \to H^s_{\varphi,h}(X)$ norm operator topology, but the convergence is not uniform as $h \to 0$.  Beals's characterization  \cite[Theorem 9.12]{ez} can be used to show that $e^{\eps Q/h} \in \Psi^{0}_\delta(X)$ for any $\delta>0$, but we will not need this. Let $s \in \R$. Then
\begin{equation}\label{e:expest}
\left\|e^{\eps Q/h}\right\| \le \sum_{j=0}^\infty \frac{(C_0\log(1/h))^j}{j!} \|Q\|^j = e^{C_0 \log(1/h)\|Q\|} = h^{-C_0 \|Q\|},
\end{equation}
where all norms are $H^s_{\varphi,h}(X) \to H^s_{\varphi,h}(X)$.

If $A \in \Psi_\delta^{m}(X)$ is bounded $H^{s+m}_{\varphi,h}(X) \to H^s_{\varphi,h}(X)$ (without needing to be multiplied by a cutoff), then, by \eqref{e:expest}, 
\begin{equation}\label{e:eepsfirst}
\|e^{\eps Q/h}A e^{-\eps Q/h}\|_{H^{s+m}_{\varphi,h}(X) \to H^s_{\varphi,h}(X)} \le C h^{-N}
\end{equation}
 for any $s \in \R$, where $N = C_0(\|Q\|_{H^{s+m}_{\varphi,h}(X) \to H^{s+m}_{\varphi,h}(X)} + \|Q\|_{H^s_{\varphi,h}(X) \to H^s_{\varphi,h}(X)})$. But, writing $\ad_Q A = [Q,A]$ and $e^{\eps Q/h}A e^{-\eps Q/h} = e^{\eps \ad_{Q}/h}A$, for any $J \in \N$ we have the Taylor expansion
\begin{equation}\label{e:tayloradj}
e^{\eps Q/h}A e^{-\eps Q/h} = \sum_{j=0}^J \frac {\eps^j}{j!} \left(\frac{\ad_{Q}}h\right)^j  A +  \frac {\eps^{J+1}}{J!}  \int_0^1(1-t)^J  e^{-\eps t\ad_{Q}/h} \left(\frac{\ad_{Q}}h\right)^{J+1} A dt.
\end{equation}
For any $M \in \N$, the integrand maps $H^{M}_{\varphi,h}(X) \to H^{-M}_{\varphi,h}(X)$ with norm $\Oh(h^{-2\delta(J+1)-N})$, $N = C_0(\|Q\|_{H^{M}_{\varphi,h}(X) \to H^{M}_{\varphi,h}(X)} + \|Q\|_{H^{-M}_{\varphi,h}(X) \to H^{-M}_{\varphi,h}(X)})$. Hence applying \eqref{e:tayloradj} with $J$ sufficiently large we see that \eqref{e:eepsfirst} can be improved to
\[
\|e^{\eps Q/h}A e^{-\eps Q/h}\|_{H^{s+m}_{\varphi,h}(X) \to H^s_{\varphi,h}(X)} \le C,
\]
and the integrand in \eqref{e:tayloradj} maps $H^{M}_{\varphi,h}(X) \to H^{-M}_{\varphi,h}(X)$ with norm $\Oh(1)$.
Applying \eqref{e:tayloradj} with $J \to \infty$ shows that $e^{\eps Q/h}A e^{-\eps Q/h} \in \Psi_\delta^m(X)$, and applying \eqref{e:tayloradj} with $J = 1$ we find
\begin{equation}\label{expexp}e^{\eps Q/h}A e^{-\eps Q/h} = A - \eps [A,Q/h] + \eps^2 h^{-4\delta} R,\end{equation}
where $R \in \Psi^{-\infty}_\delta(X)$. 

\section{Reduction to estimates for model operators}\label{s:reduce}

\subsection{Resolvent gluing}\label{s:glue} We reduce \eqref{e:main} to a series of estimates for model operators using a variant of the gluing method of \cite{Datchev-Vasy:Gluing}, adapted to the dynamics on $X$.

Let $P_C, P_K, P_F$ be \textit{model operators} for $P$ in the sense that they satisfy
\[
P_C|_{\{r < - R_g\}} = P|_{\{r < - R_g \}}, \quad P_K|_{\{|r| <  R_g + 3\}} = P|_{\{|r| <  R_g + 3\}}, \quad P_F|_{\{r > R_g\}} = P|_{\{r >  R_g\}}.
\]
So $P_C$ is a model in the cusp, $P_F$ is a model in the funnel, and $P_K$ is a model in a neighborhood of the remaining region (see Figure \ref{f:mfld}). 
We will construct the operators such that  $i(P_j  - P_j^*)= 2W_j$ for each $j \in \{C,K,F\}$, where $W_j \in C^\infty(X;[0,1])$ will be specified below. Note that  $W_j \ge 0$ implies $ \la \im P_j u, u \ra_{L^2_\varphi(X)} \le 0$ and hence
\[
 \|u\|_{L^2_\varphi(X)} \le (\im \lambda)^{-1} \|(P_j - \lambda)u\|_{L^2_\varphi(X)}, \quad \im \lambda > 0.
\]
Combining this with \eqref{e:hphidef} gives, for any $\chi_j \in C^\infty(X)$  bounded with all derivatives and satisfying $\supp \chi_j \subset \{P_j = P\}$,
\begin{equation}\label{e:modelboundj0}
\max_{j \in \{C,K,F\}}\|\chi _j R_j(\lambda) \chi_j \|_{L^2_\varphi(X) \to H^2_{\varphi,h}(X)} \le C(|\lambda| + (\im \lambda)^{-1}),  \quad \im \lambda > 0.
\end{equation}

Moreover we will construct $P_C, P_K,P_F$ such that  for every $\chi \in C_0^\infty(X)$, $E \in (0,1)$, there is $C_0 > 0$ such that for all $\Gamma > 0$ the cutoff resolvents $\chi R_j(\lambda) \chi$ continue holomorphically to $\lambda \in [-E,E] + i[-\Gamma h,\Gamma h]$, where they satisfy
\begin{equation}\label{e:modelboundj}
\max_{j \in \{C,K,F\}} \|\chi R_j(\lambda) \chi\|_{L^2_\varphi(X) \to H^2_{\varphi,h}(X)} \le C h^{-1 - C_0 |\im \lambda|/h}.
\end{equation}
Here $\chi$, $E$, $C_0$, and $\Gamma$ are the same as in \eqref{e:main}, but as elsewhere in the paper the constant $C$ and the implicit constant $h_0$ may be different.

We will also show  that the  $R_j(\lambda)$  propagate singularities forward along bicharacteristics, in the following limited sense. Let $\chi_1 \in C_0^\infty(X)$ and let $\chi_2,\chi_3 \in \Psi^1(X)$ be compactly supported differential operators. If $\supp \chi_1 \cup \supp \chi_3 \subset \{r < R_g+2\}$ and $\supp \chi_2 \subset \{r>R_g+2\}$, then, for any $N \in \N$,
\begin{equation}\label{e:propsingfk}
\|\chi_3 R_F(\lambda) \chi_2 R_K(\lambda) \chi_1\| _{L^2_\varphi(X) \to L^2_\varphi(X)}= \Oh(h^\infty),
\end{equation}
uniformly in $| \re \lambda| \le E$, $\im \lambda \in [-\Gamma h, h^{-N}]$.
If  $\supp \chi_1 \cup \supp \chi_3 \subset \{r < - R_g-2\}$ and $\supp \chi_2 \subset \{r>-R_g-2\}$, then, for any $N \in \N$,
\begin{equation}\label{e:propsingkc}
\|\chi_3 R_K(\lambda) \chi_2 R_C(\lambda) \chi_1\|_{L^2_\varphi(X) \to L^2_\varphi(X)} = \Oh(h^\infty)
\end{equation}
uniformly in$| \re \lambda| \le E$, $\im \lambda \in [-\Gamma h, h^{-N}]$. 

Note that in the first case \eqref{e:convexity} implies that no bicharacteristic passes through $T^*\supp \chi_1$, $T^*\supp\chi_2$, $T^*\supp\chi_3$ in that order, and in the second case this is implied by \eqref{e:convexity} together with the assumption that $\gamma^{-1}(\{r<-R_g\})$ is connected for any geodesic $\gamma \colon \R \to X$.  We will use these facts in the proofs of \eqref{e:propsingfk} and \eqref{e:propsingkc} below. 

Suppose for the remainder of the subsection that $P_C,P_K,P_F$ have been constructed. Let $\chi_C,\chi_K,\chi_F \in C^\infty(\R)$ satisfy $\chi_C + \chi_K + \chi_F = 1$,   $\supp \chi_F \subset (R_g + 1, \infty)$, $\supp (1-\chi_F) \subset (R_g+2,\infty)$, and $\chi_C(r) = \chi_F(-r)$ for all $r \in \R$. Then define a parametrix for $P-\lambda$ by
\[
G = \chi_C(r - 1)R_C(\lambda) \chi_C(r) + \chi_K(|r - 1|)R_C(\lambda) \chi_K(|r|) + \chi_F(r + 1)R_F(\lambda) \chi_F(r).
\]
Then $G$ is defined for $\im \lambda > 0$ and $\chi G \chi$ continues holomorphically to $\lambda \in [-E,E] - i[0,\Gamma h]$. Define operators $A_C,A_K,A_F$ by
\[\begin{split}
(P - \lambda)G &= \Id + [\chi_C(r - 1),h^2D_r^2]R_C(\lambda) \chi_C(r) +  [\chi_K(|r - 1|),h^2D_r^2] R_K(\lambda) \chi_K(|r|) \\ 
& \hspace{2.7in} + [\chi_F(r + 1),h^2D_r^2]R_F(\lambda) \chi_F(r) \\
&= \Id + A_C + A_K + A_F;
\end{split}\]
see Figure \ref{f:gluing}.
\begin{figure}[htbp]
\includegraphics[width=140mm]{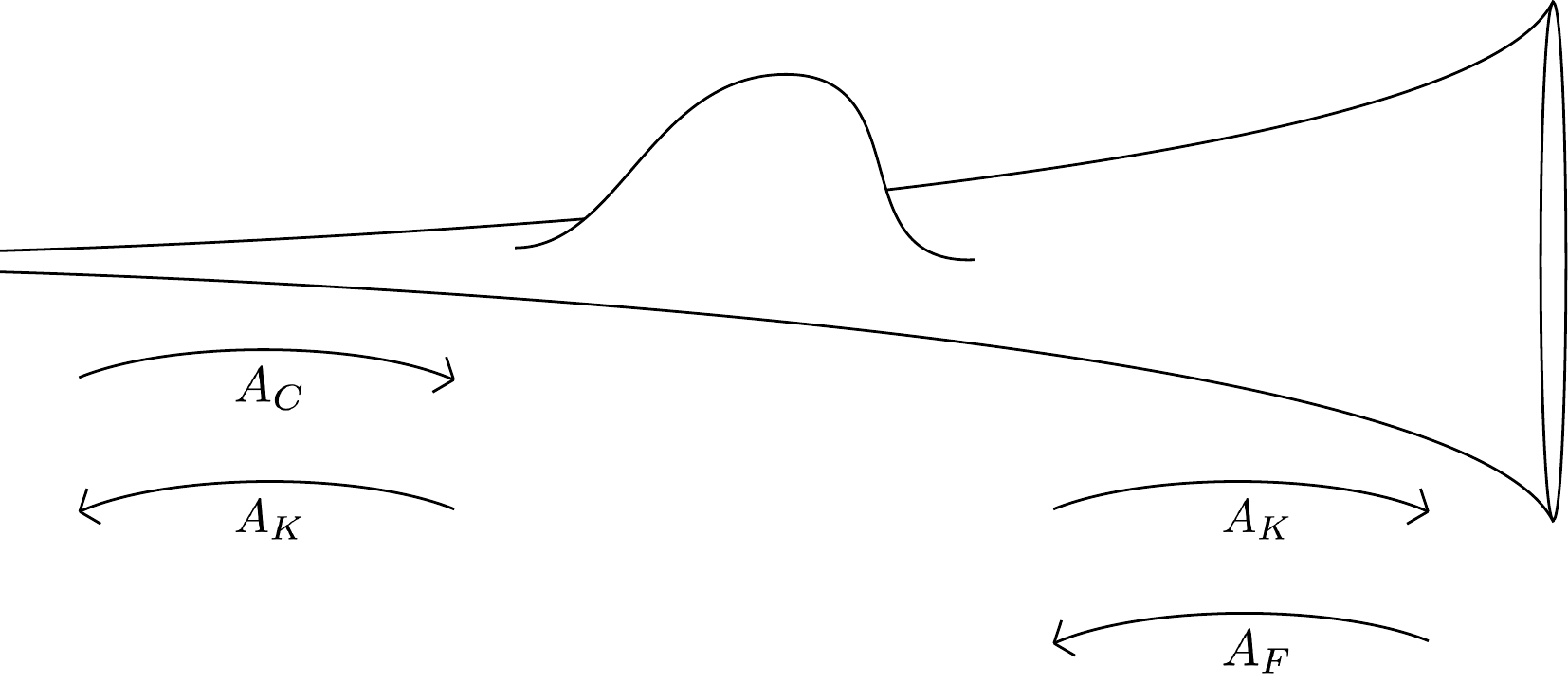}
\caption{The remainders $A_C$, $A_K$, and $A_F$ are localized on the right  in the region to the back of the arrows, and on the left near the tips of the arrows ($A_C$ is localized on the right at the support of $\chi_C$ and on the left at the support of $\chi_C'(\cdot-1)$, and so on), and this implies \eqref{e:a2}. They are microlocalized on the left in the indicated directions, and this implies \eqref{e:remtriv} (since, by \eqref{e:convexity}, no geodesic can follow one of the $A_K$ arrows and then the $A_F$ arrow, and so on).}\label{f:gluing}
\end{figure}
The estimates \eqref{e:modelboundj0} and \eqref{e:modelboundj} only allow us to remove the remainders $A_C,A_K,A_F$ by Neumann series for a narrow range of $\lambda$. To obtain improved remainders, observe that the support properties of the $\chi_j$ imply that
\begin{equation}\label{e:a2}
A_C^2 = A_K^2 = A_F^2 = A_C A_F = A_FA_C = 0;
\end{equation}
so, solving away using $G$, we obtain
\[
(P - \lambda) G(\Id - A_C - A_K - A_F) = \Id - A_KA_C - A_CA_K - A_FA_K - A_KA_F.
\]
Now the propagation of singularities estimates \eqref{e:propsingfk} and \eqref{e:propsingkc} imply
\begin{equation}\label{e:remtriv}
  \|A_FA_K \|_{L^2_\varphi(X) \to L^2_\varphi(X)} + \| A_CA_KA_CA_K\|_{L^2_\varphi(X) \to L^2_\varphi(X)} = \Oh(h^\infty),
\end{equation}
In this sense the $A_FA_K$ remainder term  is negligible. We again use \eqref{e:a2} to write
\[\begin{split}
(P - \lambda) &G(\Id - A_C - A_K - A_F +A_KA_C + A_C A_K  +  A_K A_F) = \\
&\Id  - A_FA_K  + A_CA_KA_C + A_FA_KA_C + A_KA_CA_K + A_CA_KA_F + A_KA_FA_K.
\end{split}\]
Now all remainders but  $A_CA_KA_C$, $A_KA_CA_K $, and $A_CA_KA_F$ are negligible in the sense of \eqref{e:remtriv}. Solving away again gives
\[\begin{split}
(P - \lambda) G(\Id - A_C - &A_K - A_F + A_KA_C + A_C A_K + A_K A_F \\- A_C&A_KA_C - A_KA_CA_K -  A_CA_KA_F) = \\
\Id  &- A_FA_K  + A_FA_KA_C + A_KA_FA_K \\&-A_KA_CA_KA_C -  A_CA_KA_CA_K  -  A_FA_KA_CA_K  -  A_KA_CA_KA_F.
\end{split}\]
Now all remainders but $A_KA_CA_KA_C$ are negligible. Solving away one last time gives
\[\begin{split}
(P - \lambda) G&(\Id - A_C - A_K - A_F + A_KA_C + A_C A_K + A_K A_F \\- A_C&A_KA_C - A_KA_CA_K -  A_CA_KA_F + A_KA_CA_KA_C) = \\
\Id  &- A_FA_K   + A_CA_KA_C+ A_FA_KA_C + A_KA_FA_K  -  A_CA_KA_CA_K  \\& -  A_FA_KA_CA_K  -  A_KA_CA_KA_F + A_CA_KA_CA_KA_C + A_FA_KA_CA_KA_C  = \Id + R,
\end{split}\]
where $R$ is defined by the equation, and $\|R\|_{L^2_\varphi(X) \to L^2_\varphi(X)} = \Oh(h^\infty)$. 
So for $h$ small enough we may write
\[\begin{split}
(P-\lambda)^{-1} = G\Big(&\Id - A_C - A_K - A_F + A_KA_C + A_C A_K + A_K A_F\\&  - A_CA_KA_C - A_KA_CA_K -  A_CA_KA_F + A_KA_CA_KA_C\Big)\sum_{k=0}^\infty (-R)^k.
\end{split}\]
Combining this equation with \eqref{e:modelboundj},  we see that $\chi(P-\lambda)^{-1}\chi$  continues to holomorphically to $|\re \lambda| \le E$, $\im \lambda \ge -\Gamma h$ and obeys
\[
\| \chi(P-\lambda)^{-1}\chi\|_{L^2_\varphi(X) \to H^2_{\varphi,h}(X)} \le C h^{-1 - 5C_0|\im \lambda|/h}.
\]

In summary, to prove \eqref{e:main} (and hence \eqref{logreg}), it remains to construct $P_C,P_K,P_F$ which satisfy \eqref{e:modelboundj0}, \eqref{e:modelboundj}, \eqref{e:propsingfk} and \eqref{e:propsingkc}.
We conclude this subsection by stating two Propositions which contain the estimates we will prove for $R_K(\lambda)$, after which we show how they reduce \eqref{e:propsingfk} and \eqref{e:propsingkc} to simpler propagation of singularities estimates for  $R_F(\lambda)$ and $R_C(\lambda)$ respectively, namely \eqref{e:modelpropf} and \eqref{e:modelprop}. In the next subsection we construct $P_K$ and prove the two Propositions.

\begin{prop}\label{p:rkbound}
For any $E \in (0,1)$ there is $C_0>0$ such that for any $M>0$ there are $C,h_0>0$ such that
\begin{equation}\label{e:rkbound}
\|R_K(\lambda)\|_{L^2_\varphi(X) \to H^2_{\varphi,h}(X)} \le C
\begin{cases}
h^{-1} + |\lambda|, \qquad & \im \lambda > 0, \\
h^{-1} e^{C_0 |\im \lambda|/h}, \qquad &\im \lambda \le 0,
\end{cases}
\end{equation}
for $|\re \lambda| \le E$,  $- Mh\log(1/h) \le \im \lambda $, $h \in (0,h_0]$.
\end{prop}

\begin{prop}\label{p:rkprop}
Let $\Gamma \in \R$, $E \in (0,1)$.  Let $A,B \in \Psi^0(X)$ have full symbols $a$ and $b$ with the projections to $X$ of $\supp a$ and $\supp b$  compact and suppose that 
\begin{equation}\label{e:rkpropdyn}
\supp a \cap \left[\supp b \cup \bigcup_{t \ge 0} \exp(tH_p) \left[ p^{-1}([-E,E]) \cap \supp b\right]\right] = \varnothing,
\end{equation}
where $\exp(tH_p)$ is the bicharacteristic flow of $p$, then,  for any $N \in \N$,
\begin{equation}\label{e:rkprop}
\|AR_K(\lambda)B\|_{L^2_\varphi(X) \to H^2_{\varphi,h}(X)} = \Oh(h^\infty),
\end{equation}
for  $|\re \lambda| \le E$, $ -\Gamma h \le \im \lambda \le h^{-N}$.
\end{prop}

Take $\varphi \in C^\infty(\R)$, bounded with all derivatives and supported in $(0,\infty)$, and take $\widetilde \chi_2, \ \widetilde \chi_3 \in C_0^\infty(X)$ such that $\supp \widetilde\chi_2 \subset \{r>R_g+2\}$ and $\widetilde\chi_3 \subset \{r < R_g+2\}$, and such that $\widetilde\chi_2 \chi_2 = \chi_2 \widetilde\chi_2 = \chi_2$ and $\widetilde\chi_3 \chi_3 = \chi_3 \widetilde\chi_3 = \chi_3$. 
Then \eqref{e:propsingfk} follows from
\begin{equation}\label{e:propsingbreak1}
\|\widetilde \chi_3 R_F \widetilde \chi_2 \varphi(h D_r) \| _{L^2_\varphi(X) \to H^2_{\varphi,h}(X)} + \|\widetilde \chi_2 (\Id - \varphi(hD_r)) R_K \chi_1\| _{L^2_\varphi(X) \to H^2_{\varphi,h}(X)}= \Oh(h^\infty).
\end{equation}
The estimate on the first term follows from \eqref{e:modelpropf} below, while the estimate on the second term follows from \eqref{e:rkprop} if $\supp(1-\varphi)$ is contained in a sufficiently small neighborhood of $(-\infty,0]$; it suffices to take a neighborhood small enough that no bicharacteristic in $p^{-1}([-E,E])$ goes from $T^*\supp \chi_1$ to $(T^*\supp \widetilde \chi_2) \cap \supp(1-\varphi(\rho))$, where $\rho$ is the dual variable to $r$ in $T^*X$, and such a neighborhood exists by \eqref{tanh} because when a bicharacteristic leaves $T^*\supp \chi_1$ it has $\rho \ge 0$, and \eqref{tanh} gives a minimum amount by which $\rho$ must grow in the time it takes the bicharacteristic to reach $T^*\supp \widetilde \chi_2$. An analogous argument reduces \eqref{e:propsingkc} to \eqref{e:modelprop}: the analog of \eqref{e:propsingbreak1} is
\[\|\widetilde \chi_3 R_K (\Id - \varphi(h D_r)) \widetilde \chi_2\| _{L^2_\varphi(X) \to H^2_{\varphi,h}(X)} + \| \varphi(hD_r) \widetilde \chi_2 R_C \chi_1\| _{L^2_\varphi(X) \to H^2_{\varphi,h}(X)}= \Oh(h^\infty),\]
where $\varphi \in C^\infty(\R)$ is bounded with all derivatives and supported in $(-\infty,0)$, and $\widetilde \chi_2, \ \widetilde \chi_3 \in C_0^\infty(X)$ have $\supp \widetilde\chi_2 \subset \{r>-R_g-2\}$ and $\widetilde\chi_3 \subset \{r < -R_g-2\}$, and such that $\widetilde\chi_2 \chi_2 = \chi_2 \widetilde\chi_2 = \chi_2$ and $\widetilde\chi_3 \chi_3 = \chi_3 \widetilde\chi_3 = \chi_3$.

\subsection{Model operator in the nonsymmetric region}\label{s:rk} In this subsection we define $P_K$ and prove Propositions \ref{p:rkbound} and \ref{p:rkprop}. Although the techniques involved are all essentially well known, we go over them in some detail here because they are important in the more complicated analysis of $P_C$ and $P_F$ below. 

Let $W_K \in C^\infty(X;[0,1])$ be $0$ near  $\{|r| \le R_g+3\}$, and $1$ near $\{|r| \ge R_g + 4\}$, and let 
\[
P_K = P - iW_K.
\]
We begin with the proof of Proposition 3.1, which follows \cite[\S4]{sz}. Fix
\[
E_0 \in (E,1), \qquad \eps = 10Mh\log(1/h).
\]
We will use the assumption that the flow is nontrapping to construct an \textit{escape function} $q \in C_0^\infty(T^*X)$, that is to say a function such that
\begin{equation}\label{e:escfunck}\begin{split}
 H_p q &\le -1 \textrm{ near } T^*\supp(1-W_K) \cap p^{-1}([-E_0,E_0]).
\end{split}\end{equation}
The construction will be given below. Then let $Q \in \Psi^{-\infty}(X)$ be a quantization of $q$, and
\[
P_{K,\eps} = e^{\eps Q/h}P_K e^{-\eps Q/h} = P_K - \eps [P_K,Q/h] + \eps^2 R,
\]
where $R \in \Psi^{-\infty}(X)$ (see \eqref{expexp}). We will prove that
\begin{equation}\label{e:pkepsest}
\|(P_{K,\eps} - E')^{-1}\|_{L^2_\varphi(X) \to H^2_{\varphi,h}(X)} \le 5 /\eps, \qquad E' \in [-E_0,E_0],
\end{equation}
from which it follows, using first  the openness of the resolvent set and then  \eqref{e:expest}, that
 \begin{equation}\label{e:pkest1}
\|(P_K - \lambda)^{-1}\|_{L^2_\varphi(X) \to H^2_{\varphi,h}(X)} \le \frac{h^{-N}}{M \log(1/h)}, \quad |\re \lambda| \le E_0, \ |\im \lambda| \le M h \log(1/h),
\end{equation}
where $N=10M(\|Q\|_{H^2_{\varphi,h}(X) \to H^2_{\varphi,h}(X)} + \|Q\|_{L^2_\varphi(X) \to L^2_{\varphi}(X)} )+1$.
Then we will show how to use complex interpolation to improve \eqref{e:pkest1} to \eqref{e:rkbound}.

\begin{proof}[Construction of $q \in C_0^\infty(T^*X)$ satisfying \eqref{e:escfunck}.] As in \cite[\S 4]{vz}, we  take $q$ of the form
\begin{equation}\label{e:qdefk}
q = \sum_{j=1}^J q_j,
\end{equation}
where each $q_j$ is supported near a bicharacteristic in $T^*\supp(1-W_K) \cap p^{-1}([-E_0,E_0])$.

First, for each $\wp \in T^*\supp(1-W_K) \cap p^{-1}([-E_0,E_0])$, define the following \textit{escape time}:
\[
T_\wp = \inf\{T \in \R\colon |t| \ge T-1  \Rightarrow \exp(tH_p)\wp \not\in T^*\supp(1-W_K)\}.
\]
Then put
\[
T = \max\{T_\wp\colon \wp \in T^*\supp(1-W_K) \cap p^{-1}([-E_0,E_0])\}.
\]
Note that the nontrapping assumption in \S\ref{s:assumptions} implies that $T < \infty$.  Let $\mathcal{S}_\wp$ be a hypersurface through $\wp$,  transversal to $H_p$ near $\wp$. If $U_\wp$ is a  small enough  neighborhood of $\wp$, then
\[
V_\wp = \{\exp(tH_p)\wp' \colon \wp' \in U_\wp \cap \mathcal{S}_\wp, |t| <T+1\}
\]
is diffeomorphic to $\R^{2n-1} \times (-T-1,T+1)$ with $\wp$ mapped to $(0,0)$. Denote this diffeomorphism by $(y_\wp,t_\wp)$. Further shrinking $U_\wp$ if necessary, we may  assume   the inverse image of $\R^{2n-1} \times \{|t|\ge T\}$ is disjoint from $T^*\supp(1-W_K)$. Then take $\varphi \in C_0^\infty(\R^{2n-1};[0,1])$ identically $1$ near $0$, and $\chi \in C_0^\infty((-T-1,T+1))$  with $\chi' = -1$ near $[-T,T]$, and put
\[
q_\wp = \varphi (y_\wp) \chi (t_\wp), \qquad
H_p q_\wp = \varphi(y_\wp)\chi'(t_\wp).
\]
Note $H_pq_\wp \le0$ on $T^*\supp(1-W_K)$ because $\chi' = -1$ there.
Let $V'_\wp$ be the interior of $\{H_p q_\wp = -1\}$, note that the  $V'_\wp$  cover $T^*(1-W_K) \cap p^{-1}([-E_0,E_0])$, and extract a finite subcover $\{V'_{\wp_1}, \dots, V'_{\wp_J}\}$. Then put  $q_j = q_{\wp_j}$ and define $q$ by \eqref{e:qdefk}, so that
\[
H_pq =  \sum_{j=1}^J \varphi (y_{\wp_j}) \chi_\wp' (t_{\wp_j}).
\]
Then $H_pq \le -1$ near $T^*(1-W_K) \cap p^{-1}([-E_0,E_0])$ because at each point at least one summand is, and the other summands are nonpositive.
\end{proof}

\begin{proof}[Proof of \eqref{e:pkepsest}.]
Let $\chi_0 \in C_0^\infty(X;[0,1])$ be identically $1$ on a large enough set that $\chi_0 Q = Q \chi_0 = Q$. In particular we have $(1-\chi_0) W_K = 1-\chi_0$, allowing us to write
\[
\|(1-\chi_0)u\|_{L^2_\varphi(X)}^2 = -\im \la( P_{K,\eps} - E')(1-\chi_0)u,(1-\chi_0)u\ra_{L^2_\varphi(X)}.
\]
\[
\|(1-\chi_0)u\|_{L^2_\varphi(X)} \le \|(P_{K,\eps} - E' )u\|_{L^2_\varphi(X)} + \|[P_{K,\eps},\chi_0]u\|_{L^2_\varphi(X)}.
\]
To estimate $\|\chi_0 u\|_{L^2_\varphi(X)} $ and the remainder term $\|[P_{K,\eps},\chi_0]u\|_{L^2_\varphi(X)} $ we introduce 
a microlocal cutoff $\phi \in C_0^\infty(T^*X)$ which is identically 1 near $T^*\supp(1-W_K) \cap p^{-1}([-E_0,E_0])$ and is supported in the interior of the set where $H_pq \le -1$. Since the principal symbol of $P_{K,\eps} - E'$ is
\[
p_{K,\eps} - E' = p - i W_K -E' - i\eps \{p - iW_K,q\},
\]
we have
\[
|p_{K,\eps} - E'| \ge 1-E_0, \  \textrm{ near } \supp (1 - \phi),
\]
for $|E'| \le E_0$, provided $h$ (and hence $\eps$) is sufficiently small.  Then if  $\Phi \in \Psi^{-\infty}(X)$ is a quantization of $\phi$, we find using the semiclassical elliptic estimate \eqref{ellipestx} that
\[
\|(\Id - \Phi) \chi_0 u\|_{H^2_{\varphi,h}(X)} \le C\left( \|(P_{K,\eps} - E')u\|_{L^2_\varphi(X)} + h\|u\|_{H^1_{\varphi,h}(X)}\right).
\]
Since $H_pq \le -1$ near $\supp \phi$ we see that
\[
\im p_{K,\eps} - E' = -W_K - \eps \{p,q\} \le - \eps, \ \textrm{ near } \supp \phi.
\]
Then, using the sharp G\aa rding inequality \eqref{gardingx}, we find that
\[\begin{split}
\|(P_{K,\eps} - E')\Phi  \chi_0 u\|_{L^2_\varphi(X)} \|\Phi \chi_0u\|_{L^2_\varphi(X)} &\ge - \la \im(P_{K,\eps} - E')\Phi  \chi_0u, \Phi \chi_0u\ra_{L^2_\varphi(X)} \\
&\ge \eps  \|\Phi\chi_0u\|_{L^2_\varphi(X)}^2 - Ch \|u\|_{H^{1/2}_{\varphi,h}(X)}^2.
\end{split}\]
This implies that
\[\begin{split}
\|u\|_{L^2_\varphi(X)}&\le  \|(1-\chi_0)u\|_{L^2_\varphi(X)} +  \|\Phi \chi_0 u\|_{L^2_\varphi(X)} + \|(\Id-\Phi)\chi_0 u\|_{L^2_\varphi(X)} \\
&\le C \|(P_{K,\eps} - E')u\|_{L^2_\varphi(X)}  + \eps^{-1} \|(P_{K,\eps} - E')u\|_{L^2_\varphi(X)}  + Ch^{1/2}\|u\|_{H^1_{\varphi,h}(X)},
\end{split}\]
As in the proof of \eqref{e:modelboundj0}, combining this with
\begin{equation}\label{e:pkupgrade}\begin{split}
\|u\|_{H^2_{\varphi,h}(X)} &\le 3\|u\|_{L^2_\varphi(X)} + \|(P-E')u\|_{L^2_\varphi(X)} \\&\le 4\|u\|_{L^2_\varphi(X)} + \|(P_{K,\varepsilon}-E')u\|_{L^2_\varphi(X)} + C\eps \|u\|_{L^2_\varphi(X)},
\end{split}\end{equation}
we obtain \eqref{e:pkepsest} for $h$ sufficiently small.
\end{proof}

\begin{proof}[Proof that \eqref{e:pkest1} implies \eqref{e:rkbound}.]

We follow the approach of \cite{tz} as presented in \cite[Lemma 3.1]{nsz}. Observe first that \eqref{e:modelboundj0} implies \eqref{e:rkbound} for  $\im \lambda \ge C_\Omega h$ for any $C_\Omega>0$.

Let $f(\lambda,h)$ be holomorphic in  $\lambda$  for $\lambda \in \Omega = [-E_0,E_0] + i [-Mh\log(1/h), C_\Omega h]$ and bounded uniformly in $h$  there. Suppose further that, for $\lambda \in \Omega$,
\[
|\re \lambda|  \le E \Rightarrow |f| \ge 1, \qquad |\re \lambda| \in [(E+E_0)/2, E_0] \Rightarrow |f| \le h^N.
\]

\begin{figure}[htbp]
\includegraphics[width=150mm]{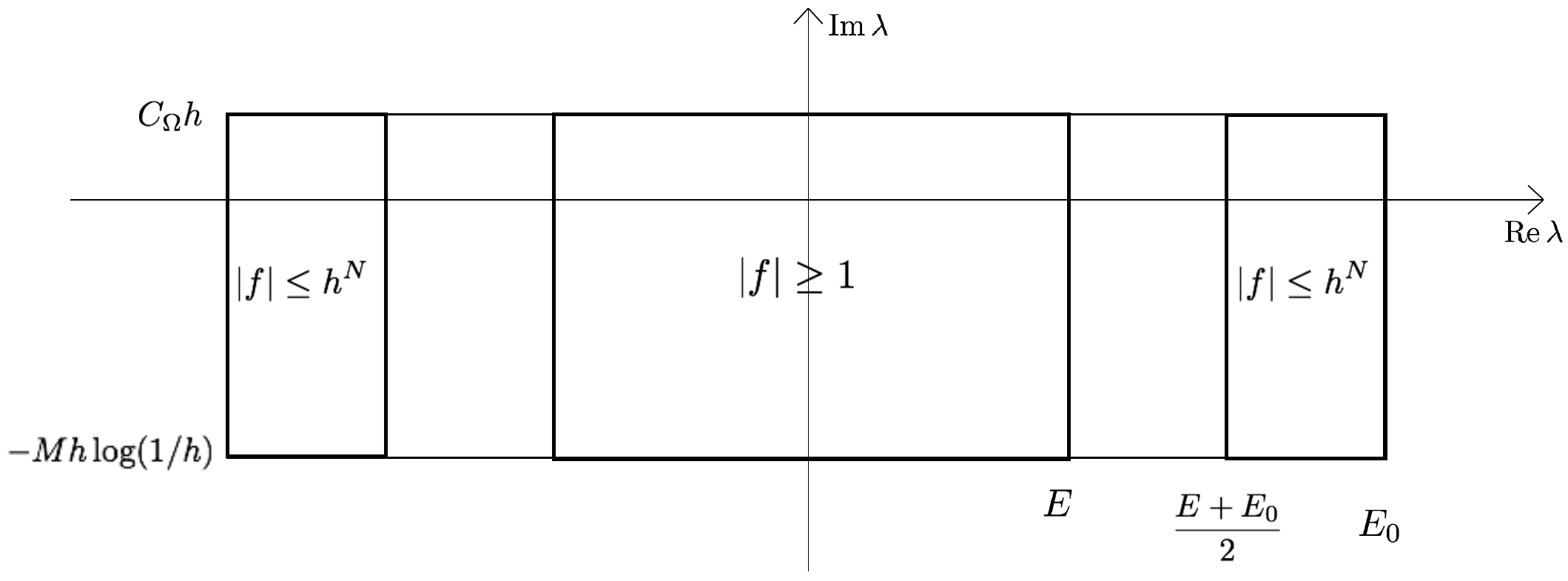}
\caption{Bounds on $f$ used in the complex interpolation argument.}
\end{figure}

For example, we may take $f$ to be a characteristic function convolved with a gaussian:
\[\begin{split}
f(\lambda, h) &= \frac 2 {\sqrt \pi} \log(1/h) \int_{-\tilde E}^{\tilde E} \exp\left(-\log^2(1/h) (\lambda - y)^2\right)dy\\
& = \erfc(\log(1/h)(\lambda - \tilde E)) - \erfc(\log(1/h)(\lambda + \tilde E)),
\end{split}\]
where $\tilde E = (3E+E_0)/4$, $\erfc z = 2 \int_z^\infty e^{-t^2}dt/\sqrt\pi$. We bound $|f|$ using the identity $\erfc(z) + \erfc(-z)= 2$ and the  fact that  $\erfc z =  \pi^{-1/2}z^{-1}e^{-z^2} (1 + \Oh(z^{-2}))$ for $|\arg z| < 3\pi/4$.

Then the subharmonic function
\[
g(\lambda,h) = \log \|(P_K -\lambda)^{-1}\|_{L^2_\varphi(X) \to H^2_{\varphi,h}(X)} + \log |f(\lambda,h)| + \frac{N \im \lambda}{Mh}
\]
obeys $g \le C$ on $\D \Omega \cap (\{|\re \lambda| = E_0\} \cup \{\im \lambda = -Mh\log(1/h)\})$, and  $g \le C + \log(1/h)$ on $\D \Omega \cap \{\im \lambda = C_\Omega h\}$. From the maximum principle and the lower bound on $|f|$ we obtain
\[
\log \|(P_K -\lambda)^{-1}\|_{L^2_\varphi(X) \to H^2_{\varphi,h}(X)} + \frac{N \im \lambda}{Mh} \le C + \log(1/h),
\]
for $\lambda \in \Omega$, $|\re \lambda| \le E$, from which \eqref{e:rkbound} follows for $\lambda \in \Omega$. 
\end{proof}

\begin{proof}[Proof of Proposition \ref{p:rkprop}]
This is similar to \cite[Lemma 5.1]{Datchev-Vasy:Gluing}. By  \eqref{ellipestx}, without loss of generality we may assume that $a$ is supported in a neighborhood of $p^{-1}([-E,E]) \cap \supp (1-W_K)$ which is as small as we please (but independent of $h$). In particular we may assume $\supp a$ is compact.

We will show that if $(P_K - \lambda)u = B f$ with $\|f\|_{L^2_\varphi(X)} = 1$, and if $\|A_0u\| \le C h^k$ for some $A_0 \in \Psi^0(X)$ with full symbol $a_0$ such that
\[
a_0 = 1 \textrm{ near } \supp a \cap p^{-1}([-E,E]),\qquad \supp a_0 \cap \bigcup_{t \ge 0} \exp(tH_p)\supp b = \varnothing,
\] 
then $\|A_1 u\| \le C h^{k+1/2}$ for each $A_1 \in \Psi^0(X)$ with full symbol $a_1$ satisfying $a_0 = 1$ near $\supp a_1$. Then the conclusion \eqref{e:rkprop} follows by induction: the base step is given by \eqref{e:rkbound}.

Let $q \in C_0^\infty(T^*X;[0,\infty))$ such that:
\begin{equation}\label{e:rkpropesc1}
a_0 = 1 \textrm{ near }\supp q, \qquad H_p (q^2) \le - (2 \Gamma + 1) q^2 \textrm{ near } \supp a_1,
\end{equation}
\begin{equation}\label{e:rkpropesc2}
H_pq \le 0 \textrm{ on } T^*\supp(1-W_K).
\end{equation}
The construction of $q$ is very similar to that of the function $q$  used in the proof of Proposition \ref{p:rkbound} above, and is also given in \cite[Lemma 5.1]{Datchev-Vasy:Gluing}. Write
\[
H_p (q^2) = -\ell^2 + r,
\]
where $\ell,r \in C_0^\infty(T^*X)$ satisfy 
\begin{equation}\label{e:derk}
\ell^2 \ge (2 \Gamma + 1)q^2, \qquad \supp r \subset \{W_K = 1\}.
\end{equation}
Let $Q,L,R \in \Psi^{-\infty}(X)$ have principal symbols $q,\ell,r$ respectively. Then
\[
i[P,Q^*Q] = -hL^*L + hR +h^2F + R_\infty,
\]
where $F \in \Psi^{-\infty}(X)$ has full symbol supported in $\supp q$ and $R_\infty \in h^\infty \Psi^{-\infty}(X)$. From this we conclude that
\begin{equation}\label{e:poscommrk}\begin{split}
\|L u\|_{L^2_\varphi(X)}^2 =  &-\frac 2 h \im \la Q^*Q P u,u \ra_{L^2_\varphi(X)}+ \la Ru,u\ra_{L^2_\varphi(X)} + h\la F u,u\ra_{L^2_\varphi(X)} + \Oh(h^\infty)\|u\|_{L^2_\varphi(X)}^2\\
=&-\frac2h \im \la Q^*Q(P_K - \lambda)u,u\ra_{L^2_\varphi(X)} - \re\la Q^*QW_K u,u\ra_{L^2_\varphi(X)} - \frac 2 h \im \lambda\|Qu\|_{L^2_\varphi(X)}^2\\
&+ \la Ru,u\ra_{L^2_\varphi(X)} + h\la F u,u\ra_{L^2_\varphi(X)} + \Oh(h^\infty)\|u\|_{L^2_\varphi(X)}^2.
\end{split}\end{equation}
We now estimate the right hand of \eqref{e:poscommrk} side term by term to prove that
\begin{equation}\label{e:pklinduc}
\|Lu\|_{L^2_\varphi(X)}^2 \le 2 \Gamma \|Qu\|_{L^2_\varphi(X)}^2 + C h \|A_0 u\|_{L^2_\varphi(X)}^2 +  \Oh(h^\infty)\|u\|_{L^2_\varphi(X)}^2,
\end{equation}
Indeed, since $\supp q \cap \supp b = \varnothing$ and since $(P_K - \lambda)u = B f $  it follows that 
\[
\la Q^*Q(P_K - \lambda)u,u\ra_{L^2_\varphi(X)} = \Oh(h^\infty)\|u\|_{L^2_\varphi(X)}^2.
\]
 Next, we write
\[
- \re\la Q^*QW_K u,u\ra_{L^2_\varphi(X)} = - \re\la W_K Q u,Q u\ra_{L^2_\varphi(X)} +\la Q^*[W_K,Q]u,u\ra_{L^2_\varphi(X)},
\]
and observe that the first term is nonpositive because $W_K \ge 0$, and the second term is bounded by $C h \|A_0 u\|_{L^2_\varphi(X)}^2$. Since $\im \lambda \ge -\Gamma h$ we have $- \frac 2 h \im \lambda\|Qu\|_{L^2_\varphi(X)}^2 \le 2 \Gamma \|Qu\|^2_{L^2_\varphi(X)}$, while since $W_K=1$ on $\supp r$ we have the elliptic estimate
\[
 \qquad  \la Ru,u\ra_{L^2_\varphi(X)} = C \|R(P_K-\lambda)u\|_{L^2_\varphi(X)} \|u\|_{L^2_\varphi(X)}  + C h\|A_0 u\|^2_{L^2_\varphi(X)}, 
\]
and the first term is $\Oh(h^\infty)\|u\|^2_{L^2_\varphi(X)}$ since $\supp r \cap \supp b = \varnothing$.
Finally $h\la F u,u\ra_{L^2_\varphi(X)}  \le Ch\|A_0u\|^2$ by inductive hypothesis, giving \eqref{e:pklinduc}.

But by \eqref{e:derk} and the sharp G\aa rding inequality  we have
\[
\la (D^*D - (2\Gamma + 1)Q^*Q)u,u\ra \ge - Ch\|A_0 u\|^2 - \Oh(h^\infty)\|u\|^2.
\]
Hence  by inductive hypothesis    we have
\[
\|Q u\|^2\le Ch^{2k+1}\|u\|^2,
\]
completing the inductive step. 
\end{proof}

\section{Model operator in the cusp}\label{s:modelcusp}
Take $W_C \in C^\infty(\R;[0,1])$ with $W_C(r) = 0$ near $r \le -R_g$, $W_C(r) = 1$ near $r \ge 0$, and let
\[
P_C =  h^2D_r^2  + e^{-2(r+\beta(r))} \Delta_{S_-} + h^2V(r) - 1 - iW_C(r),
\]
with notation as in \S\ref{spectrum}.

\begin{prop} \label{p:modelbound}
For every $\chi \in C_0^\infty(X)$, $E \in (0,1)$, there is $C_0>0$ such that for any $M>0$, there are $h_0,C>0$  such that the cutoff resolvent $\chi R_C(\lambda) \chi$ continues holomorphically from $\{\im \lambda>0\}$ to $\{|\re \lambda| \le  E$, $-Mh\log\log(1/h) \le \im \lambda \le M\}, \, h \in (0,h_0]$, and obeys
\begin{equation}\label{e:modelbound}\left\|\chi R_C(\lambda) \chi\right\|_{L^2_\varphi(X) \to H^2_{\varphi,h} (X)} \le C \begin{cases}
h^{-1} + |\lambda|, \qquad & \im \lambda > 0 \\
h^{-1-C_0 |\im \lambda|/h}, \qquad &\im \lambda \le 0,
\end{cases}. \end{equation}
\end{prop}

\begin{prop}\label{p:modelprop}
Let $r_0 <0$, $\chi_- \in C_0^\infty((-\infty,r_0))$, $\chi_+ \in C_0^\infty((r_0,\infty))$, $\varphi \in C^\infty(\R)$ supported in $(-\infty,0)$ and bounded with all derivatives, $E \in (0,1)$, $\Gamma>0$ be given. Then there exists $h_0>0$ such that
\begin{equation}\label{e:modelprop}
\left\|\varphi(hD_r)\chi_+(r)R_C(\lambda)\chi_-(r)\right\|_{L^2_\varphi(X) \to H^2_{\varphi,h}(X)} = \Oh(h^\infty),
\end{equation}
 for $|\re \lambda| \le E, \, - \Gamma h \le \im \lambda \le h^{-N}$, $h \in (0,h_0]$.
 \end{prop}

To prove these propositions we separate variables over the eigenspaces of $\Delta_{S_-}$, writing $P_C = \bigoplus_{m=0}^\infty  h^2D_r^2 + (h\lambda_m)^2 e^{-2(r + \beta(r))} + h^2 V(r) - 1 - iW_C(r),$ 
where $0 = \lambda_0 < \lambda_1  \le \cdots$ are square roots of the eigenvalues of $\Delta_{S_-}$. It suffices to prove \eqref{e:modelbound}, \eqref{e:modelprop} with $P_C$ replaced by $P(\alpha)$,  with estimates uniform in $\alpha \in \{0\} \cup [h\lambda_1,\infty)$, where
\[
P(\alpha) = h^2D_r^2 + \alpha^2e^{-2(r + \beta(r))} + h^2 V(r) - 1 - iW_C(r).
\]

\subsection{The case $\alpha = 0$}  The analysis of $(P(0) - \lambda)^{-1}$ is very similar to that of $R_K$ in \S\ref{s:rk}. The only additional technical ingredient is the method of complex scaling, which for this operator works just as in \cite{sz2,sz}.

\begin{lem}\label{l:p0}
For every $\chi \in C_0^\infty(X)$, $E \in (0,1)$, there is $C_0>0$ such that for any $M>0$, there exist $h_0,C>0$  such that the cutoff resolvent $\chi(P(0) - \lambda)^{-1}\chi$  continues holomorphically from $\{\im \lambda>0\}$ to $\{|\re \lambda| \le E$, $-Mh\log(1/h) \le \im \lambda\}, \, h \in (0,h_0]$, and obeys
\begin{equation}\label{e:modelbound0}
\left\|\chi(P(0) - \lambda)^{-1}\chi\right\|_{L^2(\R) \to H^2_h(\R)} \le C h^{-1}e^{-C_0 |\im \lambda|  /h}. \end{equation}
Let $r_0 \in \R$, $\chi_- \in C_0^\infty((-\infty,r_0))$, $\chi_+ \in C_0^\infty((r_0,\infty))$,$\varphi \in C^\infty(\R)$ supported in $(-\infty,0)$ and bounded with all derivatives, $\Gamma>0$ be given. Then there exists $h_0>0$ such that
\begin{equation}\label{e:modelprop0}
\left\|\varphi(hD_r)\chi_+(r)(P(0) - \lambda)^{-1}\chi_-(r)\right\|_{L^2(\R) \to H^2_h(\R)} = \Oh(h^\infty),
\end{equation}
 for $|\re \lambda| \le E, \, - \Gamma h \le \im \lambda \le h^{-N}$, $h \in (0,h_0]$. \end{lem}

\begin{proof}[Proof of \eqref{e:modelbound0}]

We use complex scaling to replace $P(0)$ by the complex scaled operator $P_\delta(0)$, defined below. As we will see, $P_\delta(0)$ is semiclassically elliptic for $|r|$ sufficiently large and obeys \eqref{e:modelbound0} without cutoffs.

We have
\[
P(0) = h^2 D_r^2 + h^2 V(r) -1 - iW_C(r).
\]
Fix $R>R_g$ sufficiently large that 
\begin{equation}\label{e:p0supp}
\supp \chi \cup \supp \chi_+ \cup \supp \chi_- \subset (-R,\infty).
\end{equation}
Let $\gamma \in C^\infty(\R)$ be nondecreasing and obey $\gamma(r) = 0$ for $r \ge -R$, $\gamma'(r) = \tan \theta_0$ for $r \le -R-1$ (here $\theta_0$ is as in \S\ref{s:assumptions}), and impose further that $\beta(r)$ is holomorphic near $r + i \delta \gamma(r)$ for every $r < -R$, $\delta \in (0,1)$. Below we will take $\delta \ll 1$ independent of $h$.


Now put
\[
P_\delta(0) = \frac{h^2 D_r^2}{(1 + i \delta \gamma'(r))^2} -  h \frac{\delta \gamma''(r)hD_r}{(1 + i \delta \gamma'(r))^3} + h^2V(r + i \delta \gamma(r)) - 1 - iW_C(r).
\]
If we define the differential operator with complex coefficients
\[
\widetilde P(0) = h^2 D_z^2 +  h^2 V(z) -1 - W_C(z),
\]
then we have
\begin{equation}\label{e:p0restr}
P(0) = \widetilde P(0)|_{\{z = r \colon r \in \R\}}, \qquad P_\delta(0) = \widetilde P(0)|_{\{z = r + i \delta \gamma(r) \colon r \in \R\}}.
\end{equation}
We will show that if $\chi_0 \in C^\infty(\R)$ has $\supp \chi_0 \cap \supp \gamma = \varnothing$, then
\begin{equation}\label{e:p0agree}
\chi_0(P(0) - \lambda)^{-1} \chi_0 = \chi_0(P_\delta(0) - \lambda)^{-1} \chi_0, \qquad \im \lambda >0.
\end{equation}
From this it follows that if one of these operators has a holomorphic continuation to any domain, then so does the other, and the continuations agree, so that it suffices to prove \eqref{e:modelbound0} and \eqref{e:modelprop0} with $P(0)$ replaced by $P_\delta(0)$. To prove  \eqref{e:p0agree} we will prove  that if
\[(P(0)-\lambda)u = v, \qquad (P_\delta(0)-\lambda)u_\delta =v,\]
for $v \in L^2(\R)$ with $\supp v \subset \{r \colon \gamma(r) = 0\}$, and $u,u_\delta\in L^2(\R)$, then
\[ u|_{\{r \colon \gamma(r) = 0\}} = u_\delta |_{\{r \colon \gamma(r) = 0\}}.\]
Thanks to \eqref{e:p0restr}, it suffices to show that if $\tilde u$ solves $(\widetilde P(0)- \lambda) \tilde u =v$ with $\tilde u|_{\{z = r, r \in \R\}} \in L^2(\R)$, then $\tilde u|_{\{z = r + i \delta \gamma(r), r \in \R\}}\in L^2(\R)$. For the proof of this statement we may take $\lambda$ fixed with $\re \lambda = 0$ since the general statement follows by holomorphic continuation.

Observe that for $\re z < - R$, we have
\begin{equation}\label{e:p0utildeeq}( \widetilde P(0)- \lambda)\tilde u(z) = 0.\end{equation}
We will use the WKB method to construct solutions $u_\pm$ to \eqref{e:p0utildeeq} which are exponentially growing or decaying as $\re z \to -\infty$.
Define
\[f(z) =  V(z) - (1+ \lambda)/h^2,\qquad \varphi(z) =(4f(z)f''(z) - 5f'(z)^2)(16f(z))^{-5/2}.\]
Now (see e.g. \cite[Chapter 6, Theorem 11.1]{Olver:Asymptotics}) there exist  two solutions to \eqref{e:p0utildeeq} given by
\[ u_\pm(z) = f(z)^{-1/4}e^{\pm \int_{\gamma_{z,-R}} \sqrt{f(z')}dz'}(1 + b_\pm(z)), \qquad \re z < -R,\]
taking  principal branches of the  roots and with the contour of integration $\gamma_{z,-R}$ taken from $z$ to $-R$ such that $\sqrt{\re z'}$ is monotonic along $\gamma_{z,-R}$. The functions $b_\pm$ obey
\[|b_\pm(z)| \le \exp(\max (|\varphi(z')|\colon z' \in \gamma_\pm))-1 \le Ch, \]
when $\re z> R$, where  $\gamma_+$ (resp. $\gamma_-$) is a contour from $-\infty$ to $z$ (resp. $z$ to $-R$) such that $\sqrt{\re z'}$ is monotonic along the contour. It follows that, for fixed $h$ sufficiently small,
\[|u_+(z)| \le C e^{\re z/C}, \qquad |u_-(z)| \ge C e^{-\re z/C},\]
for $\re z < - R$. Hence  $\tilde u|_{\{z = r, r \in \R\}} \in L^2(\R)$ implies that that $\tilde u$ is proportional to $u_+$. This implies that $\tilde u|_{\{z = r + i \delta \gamma(r), r \in \R\}}\in L^2(\R)$, completing the proof of \eqref{e:p0agree}.

 Fix 
\[
E_0 \in (E,1), \qquad \eps = 10 M h \log(1/h).
\]
The semiclassical principal symbol of $P_\delta(0)$ is
\begin{equation}\label{e:p0symb}
p_\delta(0) = \frac{\rho^2}{(1+ i\delta \gamma'(r))^2} - 1
= \rho^2(1 + \Oh(\delta)) - 1.
\end{equation}
In this case the escape function can be made more explicit: we  take $q \in C_0^\infty(T^*\R)$ with
\begin{equation}\label{e:hp0q0}
q(r,\rho) = - 4r \rho(1-E_0)^{-2}, \qquad H_{p_\delta(0)} q  = - 8\rho^2(1-E_0)^{-2}(1 + \Oh(\delta)),
\end{equation}
on $\{|r| \le R+1, \, |\rho| \le 2\}$.
Let $Q \in \Psi^{-\infty}(\R)$ be a quantization of $q$ and put
\[
P_{\delta,\eps}(0) = e^{\eps Q/h} P_\delta(0) e^{-\eps Q/h} = P_\delta(0) - \eps[P_\delta(0),Q/h] + \eps^2 R,
\]
where $R \in \Psi^{-\infty}(\R)$ (see \eqref{expexp}).
We will prove
\begin{equation}\label{e:modelbound0eps0}
\|(P_{\delta,\eps}(0) -E')^{-1}\|_{L^2(\R) \to H^2_h(\R)}\le 5/\eps, \qquad E' \in [-E_0,E_0]
\end{equation}
 from which it follows by \eqref{e:expest} that
\begin{equation}\label{e:modelbound001}
\|(P_\delta(0) -\lambda)^{-1}\|_{L^2(\R) \to H^2_h(\R)} \le \frac{h^{-N}}{M\log(1/h)}, \qquad |\re \lambda| \le E_0, \ |\im \lambda| \le M h \log(1/h),
\end{equation}
where $N = 10M(\|Q\|_{H^2_h(\R) \to H^2_h(\R)} + \|Q\|_{L^2(\R) \to L^2(\R)})+1$. As before we will use complex interpolation to improve \eqref{e:modelbound001} to
\begin{equation}\label{e:modelbound0020}
\|(P_\delta(0) -\lambda)^{-1}\|_{L^2(\R) \to H^2_h(\R)}  \le Ch^{-1} e^{C|\im \lambda|/h}.
\end{equation}
for $- E \le \re \lambda \le E$, $- M h \log(1/h)$. Combining  \eqref{e:p0agree} and \eqref{e:modelbound0020} gives \eqref{e:modelbound0}.

Let $\phi \in C_0^\infty(\R;[0,1])$ have $\phi(\rho) = 1$ for $|\rho|$ near $[1-E_0,1+E_0]$ and  $\supp \phi \subset\{(1-E_0)/2 < |\rho| < 2\}$. By \eqref{e:p0symb}, if $\delta$ is  small enough and $h$ is small enough depending on $\delta$, then on $\supp (1-\phi(\rho))$ we have $|p_{\delta,\eps}(0) - E'| \ge \delta(1 + \rho^2)/C$, uniformly in $E' \in [-E_0,E_0]$ and in $h$, where $p_{\delta,\eps}(0)$ is the semiclassical principal symbol of $P_{\delta,\eps}(0)$. Hence, by the semiclassical elliptic estimate \eqref{ellipestrn},
\[
\|(\Id - \phi(hD_r)) u\|_{H^2_h(\R)} \le C \delta^{-1} \|(P_{\delta,\eps}(0) - E')(\Id - \phi(hD_r)) u\|_{L^2(\R)} + \Oh(h^\infty)\|u\|_{H^{-N}_h(\R)}.
\]
On $\supp\phi(\rho)$ we use the negativity of the imaginary part of the principal symbol of $P_{\delta,\eps}(0)$. Indeed, on $\{(r,\rho)\colon \rho \in \supp \phi, \, |r| \le R+1\}$ we have, using \eqref{e:hp0q0},
\[
\im p_{\delta,\eps}(0) = \im p_{\delta}(0) + \im i\eps H_{p_{\delta,\eps}(0)} q = \frac{-2\delta \gamma'(r)\rho^2}{|1+i\delta\gamma'(r)|^4} - \frac{8\eps \rho^2}{(1-E_0)^2}(1 + \Oh(\delta)) \le - \eps,
\]
provided $\delta$ is sufficiently small. Meanwhile, on $\{(r,\rho)\colon \rho \in \supp \phi, \, |r| \ge R+1\}$ we have 
\[
\im p_{\delta,\eps}(0) = \im p_{\delta}(0) + \im i\eps H_{p_{\delta,\eps}(0)} q = \frac{-2\delta \tan\theta_0 \rho^2 }{|1+i\delta  \tan\theta_0|^4} + \Oh(\eps) \le -\delta/C,
\]
provided $h$ (and hence $\eps$) is sufficiently small.

 Then, using the sharp G\aa rding inequality \eqref{gardingrn}, we have, for $h$ sufficiently small,
\[\begin{split}
\|\varphi(h D_r)u\|_{L^2(\R)} \|(P_{\delta,\eps}(0) - E')\varphi(h D_r)u\|_{L^2(\R)} &\ge - \la\im (P_{\delta,\eps}(0) - E')\varphi(h D_r)u, \varphi(h D_r)u \ra_{L^2(\R)}  \\
& \ge \eps \|\varphi(hD_r) u\|_{L^2(\R)}^2 - C h \|u\|^2_{H^{1/2}_h(\R)}.
\end{split}\]
We deduce \eqref{e:modelbound0eps0} from this just as we did \eqref{e:pkepsest} above.

To improve \eqref{e:modelbound001} to \eqref{e:modelbound0020} we use almost the same complex interpolation argument as we did to improve \eqref{e:pkest1} to \eqref{e:rkbound}. The only difference is that in the first step we note that 
\[
\im p_\delta(0) = \frac {-2\delta\gamma'(r)}{|1+i\delta\gamma'(r)|^4} \le 0,
\]
so by the sharp G\aa rding inequality \eqref{gardingrn} we have, for some $C_\Omega>0$, $\la \im P_\delta(0) u, u \ra_{L^2(\R)} \ge -C_\Omega h \|u\|_{L^2(\R)}^2$, so that $\|(P_\delta(0) - \lambda)^{-1}\|_{L^2(\R)} \le1/ C_\Omega h$, when $\im \lambda \ge 2 C_\Omega h$.
\end{proof}

\begin{proof}[Proof of \eqref{e:modelprop0}]
Let $(P_\delta(0) - \lambda)u = f,$
where $\|f\|_{L^2(\R)} =1$,  $\supp f \subset \supp \chi_-$ and $P_\delta(0)$ is as in the proof of \eqref{e:modelbound0}. We must show that 
\begin{equation}\label{e:modelprop0conc}
\|\varphi(hD_r) \chi_+(r) u\|_{H^{2}_h(\R)} = \Oh(h^\infty);
\end{equation}
recall that the replacement of $P(0)$ by $P_\delta(0)$ is justified by \eqref{e:p0agree}. To prove \eqref{e:modelprop0conc} we use an argument by induction based on a nested sequence of escape functions.

More specifically, take
\[
q = \varphi_r(r)\varphi_\rho(\rho), \qquad 
H_{p_\delta(0)} q = 
2\rho\varphi'_r(r)\varphi_\rho(\rho) + \Oh(\delta),
\]
where $\varphi_r \in C_0^\infty(\R;[0,\infty))$ with $\supp \varphi_r \subset (r_0,\infty)$, $\varphi_r' \ge 0$ near $[r_0,R+1]$ (here $R$ is as in \eqref{e:p0supp}), $\varphi_r' > 0$ near $\supp \chi_+$. Take $\varphi_\rho \in C_0^\infty(\R;[0,\infty))$ with $\supp \varphi_\rho \subset (-\infty,0)$, $\varphi_\rho' \le 0$ near $[-2,0]$, $\varphi_\rho \ne 0$ near $\supp \varphi \cap [-2,0]$.  Impose further that $\sqrt\varphi_r, \sqrt\varphi_\rho \in C_0^\infty(\R)$, and that $\varphi'_r \ge c \varphi_r$ for $r \le R+1$, where $c >0$ is chosen large enough that $H_{p_0(\delta)} q \le -(2\Gamma +1)q$ on $\{r \le R+1, \rho \ge -2\}$: see Figure \ref{f:p0prop}.

\begin{figure}[htbp]
\includegraphics{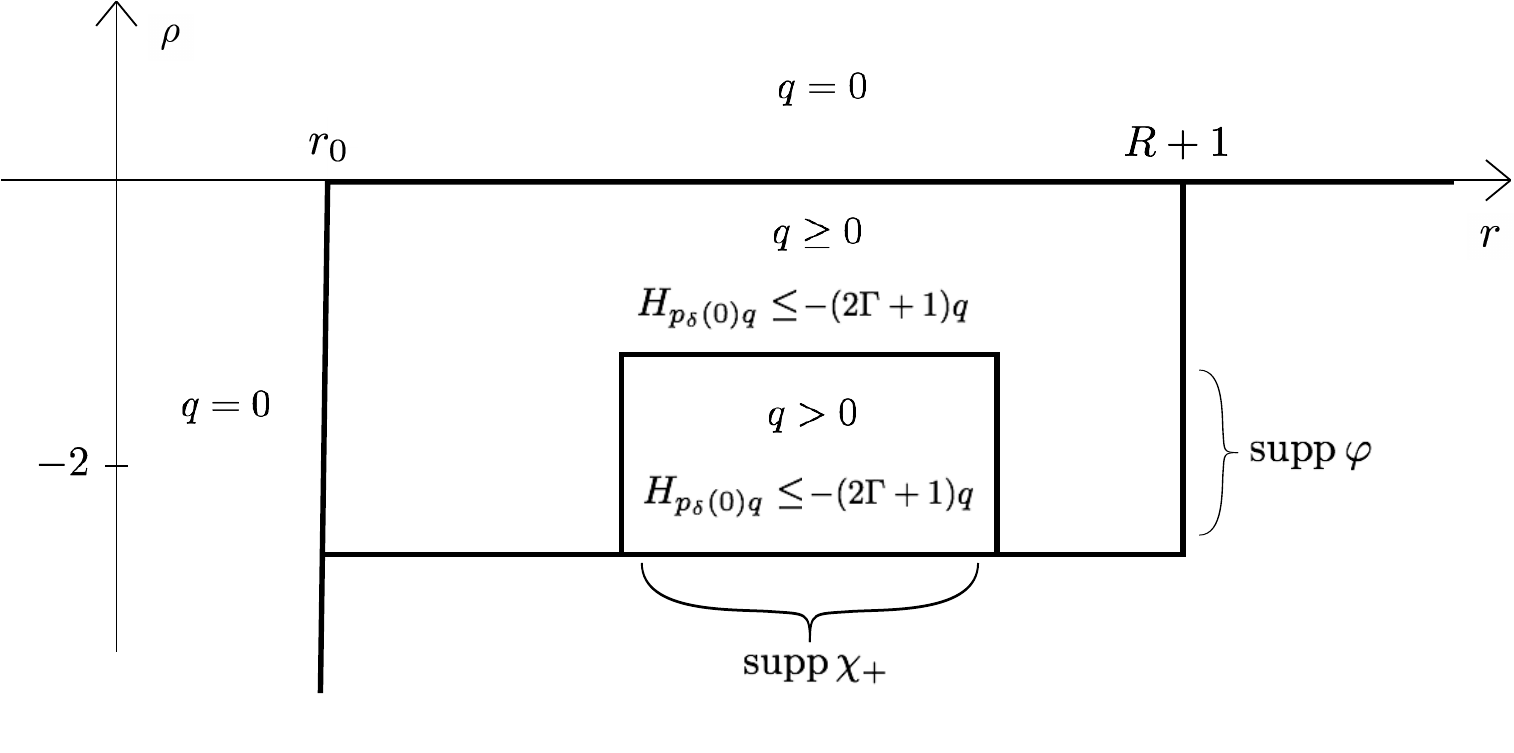}
\caption{The escape function $q$ used to prove propagation of singularities \eqref{e:modelprop0} in the case $\alpha = 0$. The derivative along the flowlines $H_{p_\delta(0)}q$ is negative and provides ellipticity for our positive commutator argument near $\{r \in \supp \chi_+, \rho \in \supp \varphi\}$. We allow $H_{p_\delta(0)}q > 0$ (the unfavorable sign for us) only in $\{r > R+1\}$ and in $\{\rho < -2\}$, because in this region $p_\delta(0)$ is elliptic.}
\label{f:p0prop}
\end{figure}

We will show that if $\|A_0u\|_{L^2(\R)} \le C h^k$ for $A_0 \in \Psi^0(\R)$ with full symbol supported sufficiently near $\supp q$ and for some $k \in \R$, then $\|A_1 u\|_{L^2(\R)} \le C h^{k + 1/2}$ for $A_1 \in \Psi^0(\R)$ with full symbol supported sufficiently near $\{ r \in \supp \chi_+, \rho \in \supp \varphi\}$. The conclusion \eqref{e:modelprop0conc} then follows by induction. (The base step of the induction follows from \eqref{e:modelbound0020} or even from \eqref{e:modelbound001}.)

In the remainder of the proof all norms and inner products are in $L^2(\R)$ and we omit the subscript for brevity.

We write
\[
H_{p_\delta(0)} q^2 = -b^2 + e,
\]
where $b,e \in C_0^\infty(T^*\R)$, $b > 0$ near $ \{ r \in \supp \chi_+, \rho \in \supp \varphi, -2 \le \rho\}$, $b^2 \ge (2\Gamma+1)q^2$ everywhere, and $\supp e \cap (\{ r  \le R+1, \rho \ge -2\} \cup \{r \le r_0\})= \varnothing$. Let $Q,B,E$ be quantizations of $q,b,e$ respectively. Then
\[
i[P_\delta(0),Q^*Q] = - hB^*B + hE + h^2F,
\]
where $F \in \Psi^0(\R)$ has full symbol supported in $\supp q$. From this we conclude that
\[
\|Bu\|^2 = - \frac2 h  \im \la Q^*Q(P_\delta(0) - \lambda)u,u\ra - \frac 2 h \im \lambda\|Q u\|^2+ \la E u, u \ra + h \la Fu,u\ra + \Oh(h^\infty)\|u\|^2.
\]
From $(P_\delta(0) - \lambda)u = f$ and $\WF'_h Q \cap T^*\supp f = \varnothing$ it follows that the first term is $\Oh(h^\infty)\|u\|^2$. Similarly $\WF'_h E \cap (\supp f\cup p_\delta^{-1}(0)) = \varnothing$ implies by \eqref{ellipestrn} that the third term is $\Oh(h^\infty)\|u\|^2$. The fourth term is bounded by $C h^{2k+1}\|u\|^2$ by inductive hypothesis, giving
\[
\|Bu\|^2 \le  2\Gamma \|Q u\|^2 + C h^{2k+1}\|u\|^2.
\]
By \eqref{gardingrn} we have
\[
\la(B^*B - (2\Gamma +1)Q^*Q)u,u\ra \ge -Ch\|Ru\|^2,
\]
where $R \in \Psi^{0,0}_0(\R)$ is microsupported in an arbitrarily small neighborhood of $\WF'_hQ$. Hence $\|Ru\| \le Ch^k \|u\|$ and we have
\[
\|Qu\|^2 \le  C h^{2k+1}\|u\|^2,
\]
completing the inductive step and also the proof.
\end{proof}

\subsection{The case $\alpha \ge \lambda_1h$.}

Propositions \ref{p:modelbound} and \ref{p:modelprop} follows from \eqref{e:modelbound0},  \eqref{e:modelprop0} and the following two Lemmas.

\begin{lem}\label{l:modelbound}
For any $E \in (0,1)$ there is $C_0>0$ such that for any $M, \lambda_1>0$ there are $h_0,C>0$  such that if $h \in (0,h_0], \alpha \ge \lambda_1 h$, $\lambda \in [-E,E] +i [-Mh\log\log(1/h),\infty)$, then
\begin{equation}\label{e:modelbound2}
\left\|(P(\alpha) - \lambda)^{-1}\right\|_{L^2(\R) \to H^2_h(\R)} \le C \log(1/h) h^{-1-C_0 |\im \lambda| /h}. \end{equation}
If $\chi \in C^\infty(\R)$ has $\chi' \in C_0^\infty(\R)$ and $\chi(r) = 0$ for $r$ sufficiently negative, then
\begin{equation}\label{e:modelbound2b}
\left\|\chi(P(\alpha) - \lambda)^{-1}\chi\right\|_{L^2(\R) \to H^2_h(\R)} \le C h^{-1-2C_0 |\im \lambda| /h}\end{equation}
in the same range of $h,\alpha,\lambda$, and with the same $C_0$ and $h_0$ (but with different $C$).
\end{lem}

\begin{lem}
Let $r_0 <0$, $\chi_- \in C_0^\infty((-\infty,r_0))$, $\chi_+ \in C_0^\infty((r_0,\infty))$, $\varphi \in C_0^\infty((-\infty,0))$, $E \in (0,1)$, $\Gamma, \lambda_1, N>0$ be given. Then there exists $h_0>0$ such that
\begin{equation}\label{e:modelpropa}
\left\|\varphi(hD_r)\chi_+(r)(P(\alpha) - \lambda)^{-1}\chi_-(r)\right\|_{L^2(\R) \to H^2_h(\R)} = \Oh(h^\infty),
\end{equation}
uniformly for $\alpha \ge \lambda_1 h$, $\re \lambda \in [-E,E], \, -\Gamma h \le \im \lambda \le h^{-N}$, $h \in (0,h_0]$.
\end{lem}

Take $\alpha_0>0$ such that if $\alpha \ge \alpha_0$ and $r \le 0$ then $\alpha^2e^{-2(r+\beta(r))} \ge 3$.
We consider the cases $\lambda_1 h \le \alpha \le \alpha_0$ and $\alpha_0 \le \alpha$ separately.

\begin{proof}[Proof of \eqref{e:modelbound2}, \eqref{e:modelbound2b}, and \eqref{e:modelpropa} for $\alpha_0 \le \alpha$] In this case $P(\alpha)$ is `elliptic' (although not pseudodifferential in the usual sense because of the exponentially growing term $\alpha^2 e^{-2(r+\beta(r))}$) and  better estimates hold.
Use the fact that $W_C \ge 0$  and $\alpha^2 e^{-2(r+\beta(r))} \ge 3$ for $r \le 0$ to write
\begin{align*}
\int_{-\infty}^0 |u|^2 dr &\le \frac 13 \int_{-\infty}^\infty \alpha^2 e^{-2(r+\beta(r))} |u|^2 dr \le \frac 13 \re \langle P(\alpha)u,u\rangle_{L^2(\R)} + \left(\frac 13 + \Oh(h^2)\right) \|u\|_{L^2(\R)}^2, \\
\int_0^\infty |u|^2dr &= \int_0^\infty W_C |u|^2dr \le \int_{-\infty}^\infty W_C|u|^2 dr = -\im \langle P(\alpha) u,u\rangle_{L^2(\R)}.
\end{align*}
Adding the inequalities gives
\[
\|u\|^2_{L^2(\R)} \le 2 \|(P(\alpha) - \lambda)u\|_{L^2(\R)}  \|u\|_{L^2(\R)} +\left(\frac 13\re \lambda - \im \lambda+ \frac 13 + \Oh(h^2)\right)\|u\|^2_{L^2(\R)}, 
\]
So long as $\im \lambda - (1/3) \re \lambda + 2/3 \ge \epsilon$ for some $\epsilon>0$, it follows that
\begin{equation}\label{e:alphabig}
\|u\|_{L^2(\R)} \le C\|(P(\alpha) - \lambda)u\|_{L^2(\R)}.
\end{equation}
To obtain \eqref{e:modelbound2} we observe that
\[\begin{split}
\|h^2D_r^2u\|_{L^2(\R)}^2= &\|(h^2D_r^2 + \alpha^2e^{-2(r+\beta(r))}) u\|_{L^2(\R)}^2 - \|  \alpha^2e^{-2(r+\beta(r))}u\|_{L^2(\R)}^2 \\&- 2 \re \langle h^2 D_r^2 u,  \alpha^2e^{-2(r+\beta(r))} u\rangle_{L^2(\R)},
\end{split}\]
while 
\[\begin{split}
- \re \langle & h^2 D_r^2 u,  \alpha^2e^{-2(r+\beta(r))}  u\rangle_{L^2(\R)} = \\& - \|\alpha e^{-(r+\beta(r))}hD_ru\|_{L^2(\R)}^2 +2 \im \langle h D_r u,  (1+\beta'(r))h\alpha^2e^{-2(r+\beta(r))} u\rangle_{L^2(\R)},
\end{split}\]
so that
\[
\|h^2D_r^2u\|_{L^2(\R)} \le 2 \|(h^2D_r^2 + \alpha^2e^{-2(r+\beta(r))}) u\|_{L^2(\R)} \le 2\|(P(\alpha) - \lambda)u\|_{L^2(\R)} + C |\lambda| \|u\|_{L^2(\R)}. 
\]
Together with \eqref{e:alphabig}, this implies \eqref{e:modelbound2} (and hence \eqref{e:modelbound2b}) with the right hand side replaced by $C(1+|\lambda|)$. The estimate \eqref{e:modelpropa} follows from the stronger Agmon estimate
\[
\left\|\chi_+(r)(P(\alpha) - \lambda)^{-1}\chi_-(r)\right\|_{L^2(\R) \to H^2_h(\R)} = \Oh(e^{-1/(Ch)}),
\]
see for example \cite[Theorems 7.3 and 7.1]{ez}.
\end{proof}

\begin{proof}[Proof of \eqref{e:modelbound2} for $\lambda_1 h \le \alpha \le \alpha_0$]
For this range of $\alpha$ we use the following rescaling (I'm very grateful to Nicolas Burq for suggesting this rescaling):
\begin{equation}\label{e:tildevars}
\tilde r = r/ \log(2\alpha_0/\alpha), \qquad \tilde h = h/\log(2\alpha_0/\alpha).
\end{equation}
In these variables we have
\[
P(\alpha) =  (\tilde hD_{\tilde r})^2 + 4\alpha_0^2 e^{-2\left[(1+\tilde r)\log(2\alpha_0/\alpha) + \tilde \beta(\tilde r)\right]}+ \tilde h^2 \widetilde V(\tilde r) - 1 - i\widetilde W_C(\tilde r),
\]
where
\[
\tilde \beta(\tilde r) = \beta(r), \qquad \widetilde V(\tilde r) = \log(2\alpha_0/\alpha)^2 V(r), \qquad \widetilde W_C(\tilde r) =W_C(r).
\]
We will show that
\begin{equation}\label{e:modelbound2tilde0}
\|(P(\alpha) - \lambda)^{-1}\|_{L^2_{\tilde r} \to H^2_{h,\tilde r}} \le C \tilde h^{-1}e^{C_0 |\im \lambda| /\tilde h},
\end{equation}
for $|\re \lambda| \le E, \ \im \lambda \ge - M \tilde h \log (1/\tilde h)$, from which \eqref{e:modelbound2} follows.

We now use a variant of the gluing argument in \S\ref{s:glue} to replace the exponentially growing term $4\alpha_0^2 e^{-2\left[(1+\tilde r)\log(\alpha_0/\alpha) + \tilde \beta(\tilde r)\right]}$ with a bounded one.  Fix $\widetilde R >0$ such that  
\[
\tilde r \le- \widetilde R,\ \alpha \le \alpha_0 \Longrightarrow \alpha_0^2 e^{-2\left[(1+\tilde r)\log(2\alpha_0/\alpha) + \tilde \beta(\tilde r)\right]} > 1.
\]
Take $\widetilde V_B, \widetilde V_E \in C^\infty(\R, [0,\infty))$ such that $\widetilde V_E(\tilde r) = 4\alpha_0^2 e^{-2\left[(1+\tilde r)\log(2\alpha_0/\alpha) + \tilde \beta(\tilde r)\right]} $ for $\tilde r \le -  \widetilde R$ and $\widetilde V_E(\tilde r) \ge 4$ for all $\tilde r$, while $\widetilde V_B(\tilde r) = 4\alpha_0^2 e^{-2\left[(1+\tilde r)\log(2\alpha_0/\alpha) + \tilde \beta(\tilde r)\right]} $ for $\tilde r \ge- \widetilde R-3$ and is bounded, uniformly in $\alpha$, together with all derivatives (see figure \ref{f:vcve}).

\begin{figure}[htbp]
\includegraphics[width=12cm]{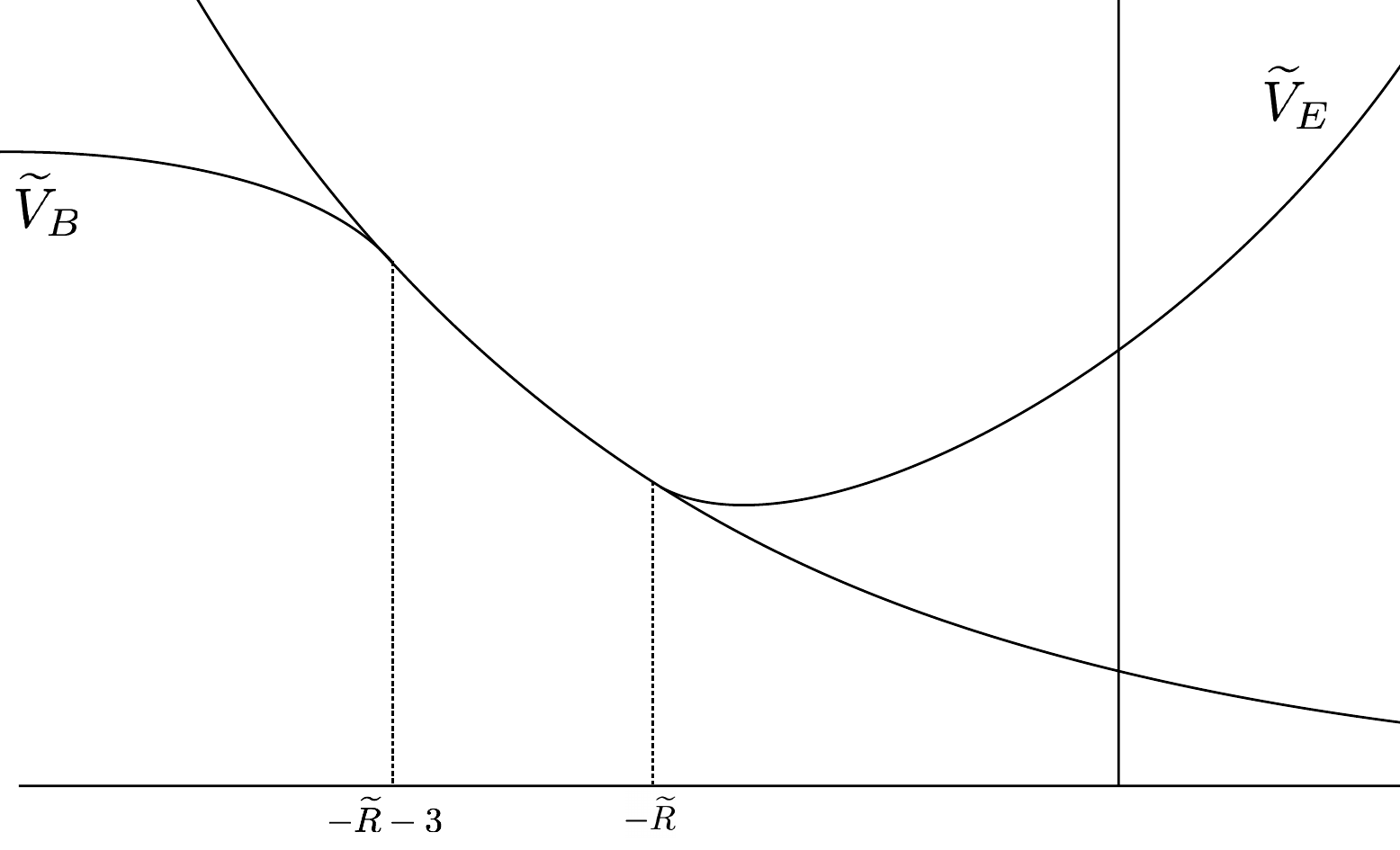}
\caption{The model potentials $\widetilde V_E$ and $\widetilde V_B$, which agree with $4\alpha_0^2 e^{-2\left[(1+\tilde r)\log(2\alpha_0/\alpha) + \tilde \beta(\tilde r)\right]}$ for $\tilde r \le- \widetilde R$  and  $\tilde r \ge - \widetilde R-3$ respectively.}\label{f:vcve}
\end{figure}

 Let
\[\begin{split}
P_E(\alpha) &=  (\tilde hD_{\tilde r})^2 + \widetilde V_E(\tilde r) + \tilde h^2 \widetilde V(\tilde r) - 1 - i\widetilde W_C(\tilde r), \\ P_B(\alpha) &=  (\tilde hD_{\tilde r})^2 + \widetilde V_B(\tilde r) + \tilde h^2 \widetilde V(\tilde r) - 1 - i\widetilde W_C(\tilde r),
\end{split}\]
and let $R_E = (P_E(\alpha)-\lambda)^{-1}$, $R_B = (P_B(\alpha)-\lambda)^{-1}$. Note that 
\[
\|R_E\|_{L^2_{\tilde r} \to H^2_{h,\tilde r}} \le C
\]
by the same proof as that of  \eqref{e:modelbound2} for $\alpha \ge \alpha_0$.
 We will show that \eqref{e:modelbound2tilde0} follows from
\begin{equation}\label{e:modelbound2tilde}
\|R_B\|_{L^2_{\tilde r} \to H^2_{h,\tilde r}} \le C \tilde h^{-1}e^{C_0 |\im \lambda| /\tilde h},
\end{equation}
for $|\re \lambda| \le E, \im \lambda \ge - M \tilde h \log (1/\tilde h)$. Indeed, let $\chi_E \in C^\infty(\R;\R)$ have $\chi_E(\tilde r) = 1$ near $\tilde r \le- \widetilde R-2$ and  $\chi_E(\tilde r) = 0$ near  $\tilde r \ge- \widetilde R-1$, and let $\chi_B = 1 - \chi_E$. Let
\[
G = \chi_E(\tilde r - 1) R_E \chi_E(\tilde r) +  \chi_B(\tilde r + 1) R_B \chi_B(\tilde r).
\]
Then
\[
(P(\alpha) - \lambda) G = \Id + [\tilde h^2D_{\tilde r}^2,\chi_E(\tilde r - 1)]  R_E \chi_E(\tilde r) +  [\tilde h^2D_{\tilde r}^2,\chi_B(\tilde r + 1)] R_B \chi_B(\tilde r) = \Id + A_E + A_B.
\]
As in \S\ref{s:glue} we have $A_E^2 = A_B^2 = 0$. We also have the Agmon estimate
\[
\|A_E\|_{L^2_{\tilde r} \to L^2_{\tilde r}} \le e^{-1/(C\tilde h)};
\]
see for example \cite[Theorems 7.3 and 7.1]{ez}. Solving away $A_B$ using $G$ we find that
\begin{equation}\label{e:palphapar}
(P(\alpha) - \lambda) G (\Id - A_B) = \Id + \Oh_{L^2_{\tilde r}\to L^2_{\tilde r}}(e^{-1/(C\tilde h)}),
\end{equation}
and since $\|G (\Id - A_B)\|_{L^2_{\tilde r} \to H^2_{\tilde h,\tilde r}} \le C \tilde h^{-1} e^{C|\im \lambda|/\tilde h}$, this implies \eqref{e:modelbound2tilde0}.

The proof of \eqref{e:modelbound2tilde} follows  that of \eqref{e:modelbound0} with these differences: the $-i\widetilde W_C(\tilde r)$ term removes the need for complex scaling, and the $\widetilde V_B(\tilde r)$ term  puts $P_B$ in a mildly exotic operator class and leads to a slightly modified escape function $q$ and microlocal cutoff $\phi$. Fix
\begin{equation}\label{e:e0cusp}
E_0 \in (E,1), \qquad \eps = 10M\tilde h \log(1/\tilde h).
\end{equation}
The $\tilde h$-semiclassical principal symbol of $P_B$ (note that $P_B \in \Psi^2_\delta(\R)$ for any $\delta>0$) is
\begin{equation}\label{e:pcsymb}
p_B = \tilde \rho^2 + \widetilde V_B(\tilde r) - 1 - i \widetilde W_C(\tilde r),
\end{equation}
where $\tilde \rho$ is dual  to $\tilde r$. Take $q \in C_0^\infty(T^*\R)$ such that on $\{-\widetilde R \le \tilde r \le 0, \ |\tilde \rho| \le 2\}$ we have
\[
q(\tilde r, \tilde \rho) = - C_q (\tilde r + \widetilde R+1) \tilde \rho, 
\]
\[
\re H_{p_B}q = -2C_q \tilde \rho^2 + C_q(\tilde r  + \widetilde R+1) \widetilde V'_B(\tilde r) \le -C_q (\re p_B + 1)
\]
where $C_q>0$ is a large constant which will be specified below, and where for the inequality we used \eqref{e:betahalf}.
Let $Q \in \Psi^{-\infty}(\R)$ be a quantization of $q$ with $\tilde h$ as semiclassical parameter and put
\begin{equation}\label{e:expexp}
P_{B,\eps} = e^{\eps Q/\tilde h} P_B e^{-\eps Q/ \tilde h} = P_B - \eps[P_B,Q/\tilde h] + \eps^2 \tilde h^{-4\delta} R,
\end{equation}
where $R \in \Psi_\delta^{-\infty}(\R)$ by \eqref{expexp}. The $\tilde h$-semiclassical principal symbol of $P_{B,\eps}$ is
\[
p_{B,\eps} = \tilde \rho^2 + V_B(\tilde r) - 1 - i \widetilde W_C(\tilde r) + i \eps H_{p_B} q
\]
We will prove
\begin{equation}\label{e:modelbouncd0eps}
\|(P_{B,\eps} -E')^{-1}\|_{L^2_{\tilde r} \to H^2_{\tilde h, \tilde r} } \le 5/\eps, \qquad E' \in [-E_0,E_0],
\end{equation}
from which it follows by \eqref{e:expest} that
\begin{equation}\label{e:modelbound00}
\|(P_{B,\eps} -\lambda)^{-1}\|_{L^2_{\tilde r} \to H^2_{\tilde h, \tilde r} }  \le \frac{\tilde h^{-N}}{M\log(1/\tilde h)}, \quad |\re \lambda| \le E_0, \ |\im \lambda| \le M \tilde h \log(1/\tilde h)
\end{equation}
where $N = 10M(\|Q\|_{H^2_{\tilde h, \tilde r} \to H^2_{\tilde h, \tilde r}} + \|Q\|_{L^2_{\tilde r} \to L^2_{\tilde r}}) + 1$. The proof that \eqref{e:modelbound00} implies \eqref{e:modelbound2tilde} is the same as the proof that \eqref{e:pkest1} implies \eqref{e:rkbound}.

Let $\phi \in C_0^\infty(T^*\R)$ be identically $1$ near $\{(\tilde r, \tilde \rho): -\widetilde R \le \tilde r \le 0,\ |\tilde \rho|\le2,\ |\re p_B(\tilde r, \tilde \rho)| \le E_0\}$ and be supported such that $\re H_{p_B}q <0$ on $\supp \phi$. Let $\Phi$ be the quantization of $\phi$ with $\tilde h$ as semiclassical parameter. For $h$ (and hence $\tilde h$ and $\eps$) small enough,  we have $|p_{B,\eps} - E'| \ge (1 + \tilde \rho^2)/C$ on $\supp (1-\phi)$, uniformly in $E' \in [-E_0,E_0]$, in $\alpha \le \alpha_0$ and in $h$. Hence, by the semiclassical elliptic estimate \eqref{ellipestrn},
\[
\|(\Id - \Phi) u\|_{H^2_{\tilde h, \tilde r}} \le C \|(P_{B,\eps} - E')(\Id - \Phi) u\|_{L^2_{\tilde r}} + \Oh(h^\infty)\|u\|_{H^{-N}_{\tilde h, \tilde r}}.
\]
Using the fact that $\re H_{p_B}q <0$ on $\supp\phi$, fix $C_q$ large enough that on $\supp\phi$  we have
\[
\im p_{B,\eps} = -  \widetilde W_C(\tilde r) + \eps \re H_{p_B} q \le - \eps.
\]
Then, using the sharp G\aa rding inequality \eqref{gardingrn}, we have, for $h$ sufficiently small,
\[\begin{split}
\|\Phi u\|_{L^2_{\tilde r}(\R)} \|(P_{B,\eps} - E')\Phi u\|_{L^2_{\tilde r}(\R)} &\ge - \la\im (P_{B,\eps} - E')\Phi u, \Phi u \ra_{L^2_{\tilde r}(\R)}  \\
& \ge \eps \|\Phi u\|_{L^2_{\tilde r}(\R)}^2 - C \tilde h^{1-2\delta} \|u\|^2_{H^{1/2}_{\tilde h, \tilde r}(\R)}.
\end{split}\]
We deduce \eqref{e:modelbouncd0eps} from this just as we did \eqref{e:pkepsest} above.
\end{proof}

\begin{proof}[Proof of \eqref{e:modelbound2b} for $\lambda_1 h \le \alpha \le \alpha_0$.]
It suffices to show that
\begin{equation}\label{e:modelboundcomp}
\left\|\chi R_B \chi\right\|_{L^2_r \to H^2_{h,r}} \le C/h, \end{equation}
when $|\re \lambda| \le E_0, \,\im \lambda \ge 0$, with $R_B$ as in the proof of \eqref{e:modelbound2} for $\lambda_1 h \le \alpha \le \alpha_0$, $E_0$ as in \eqref{e:e0cusp}\footnote{Note that for this proof we do not use the variables $\tilde r$ and $\tilde h$.}. Then $\left\|\chi P(\alpha) - \lambda)^{-1} \chi\right\|_{L^2_r \to H^2_{h,r}} \le C/h$ (for the same range of parameters) follows by the same argument that reduced \eqref{e:modelbound2}  to \eqref{e:modelbound2tilde} above. After this, \eqref{e:modelbound2b}  follows by complex interpolation as in  the proof that \eqref{e:pkest1} implies \eqref{e:rkbound} above. Indeed,  take $f(\lambda,h)$ holomorphic in $\lambda$, bounded uniformly for $\lambda \in \Omega = [-E_0,E_0]  + i [-Mh\log\log(1/h),0]$, and satisfying
\[
|\re \lambda| \le E \Rightarrow |f| \ge 1, \qquad |\re \lambda| \le [(E+E_0)/2,E_0] \Rightarrow |f| \le h^2
\]
for $\lambda \in \Omega$. Then define the subharmonic function
\[
g(\lambda,h) = \log\|\chi (P(\alpha) - \lambda)^{-1}\chi\|_{L^2_r \to H^2_{h,r}} + \log|f(\lambda,h)| + 2C_0 \frac{\im \lambda} h \log(1/h),
\]
and apply the maximum principle to $g$ on $\Omega$, observing that $g \le C + \log(1/h)$ on $\D \Omega$.

 It now remains to prove \eqref{e:modelboundcomp}, which we do using a `non-compact' variant of the positive commutator method of \cite{Datchev-Vasy:Propagation}. Fix $-R_0 < \inf \supp \chi$ and take $f \in L^2_r$ with  $\supp f \subset (-R_0,\infty)$. Let $u = R_B f$. We will show that $\|\chi u\|_{H^2_{h,r}} \le C\|f\|_{L^2_r}/h$.
 
As an escape function take $q \in S^0(\R)$ with $q\ge 0$ everywhere and such that
\[
q(r,\rho) = \begin{cases} 1 + 2 R_0e^{-1/R_0}, & -R_0 \ge r, \\ 1 + 2 R_0e^{-1/R_0} -\rho (r+R_0+1) e^{-1/(r+R_0)}, & -R_0 < r \le 0 \textrm{ and } |\rho|\le2. \end{cases}
\]
We do not prescribe additional conditions on $q$ outside of this range of $(r,\rho)$, as $P_B$ is semiclassically elliptic there. The $h$-semiclassical principal symbol of $P_B$ is (see \eqref{e:pcsymb})
\[
p_B = \rho^2 +  V_B( r) - 1 - i W_C( r),
\]
where $V_B(r) = \widetilde V_B(\tilde r)$. 
Making $-\widetilde R$ more negative if necessary, we may suppose without loss of generality that
\[
r \ge -R_0 \Longrightarrow V_B(r) = \alpha^2 e^{-2(r + \beta(r))}.
\]
For $r \le -R_0$ we have $H_{p_B} q = 0$,  and for $-R_0 < r \le 0$, $|\rho| \le 2$ we have
\[\begin{split}
\re H_{p_B} q(r,\rho) &= \left[-2\rho^2 (1 + 1/(r+R_0) ) + V'_B(r)(r+R_0+1) \right]e^{-1/(r+R_0)} \\
&\le -(\re p_B +1)e^{-1/(r+R_0)}.
\end{split}\]
Consequently we may write
\[
\re H_{p_B} (q^2) = -b^2 + a,
\]
where $a,b \in C^\infty_0(T^*\R)$ and $\supp a$ is disjoint from $\{r \le -R_0\}$ and from $\{-R_0<r\le 0\} \cap \{|\rho|\le2\}$. Note that 
\begin{equation}\label{e:bne0pc}
b \ne 0 \textrm{ on } \{|p_B| \le E_0\} \cap T^*(-R_0,0).
\end{equation}
Let $Q = \Op(q)$ as in \eqref{quantdef}.
Then
\begin{equation}\label{e:pcposcom}
i[P_B,Q^*Q] = -h B^*B + hA + [W_C,Q^*Q] + h^2Y,
\end{equation}
where $B,A,Y \in \Psi^{-\infty}(\R)$ and $B,A$ have semiclassical principal symbols $b,a$. Note that if $\chi_0 \in C_0^\infty((-R_0,\infty))$, then by \eqref{e:bne0pc} and \eqref{ellipestrn} we have
\begin{equation}\label{e:chi0pc}
\| \chi_0 u\|^2_{H^2_{h,r}} \le C (\|Bu\|^2_{L^2_r} + \log^2(1/h)\|f\|^2_{L^2_r}),
\end{equation}
so it suffices to show that
\begin{equation}\label{e:bpc}
\|Bu\|^2_{L^2_r} \le  C h^{-2} \|f\|^2_{L^2_r}.
\end{equation}
Combining \eqref{e:pcposcom} with
\[
 \langle i[P_B,Q^*Q]u,u\rangle_{L^2_r}  
 = -2\im\langle Q^*Qu,f\rangle_{L^2_r} +2\langle W_C Q^*Q u,u\rangle_{L^2_r} + 2 \im \lambda \|Q u\|^2_{L^2_r}
\]
gives
\begin{equation}\label{e:bpcterms}\begin{split}
\|Bu\|^2_{L^2_r} &= \langle Au,u\rangle_{L^2_r} + \frac 2 h\im\langle Q^*Qu,f\rangle_{L^2_r} - \frac 1 h\langle (W_C Q^*Q  + Q^*Q W_C)u,u\rangle_{L^2_r} \\
&\qquad\qquad - \frac{2 \im \lambda}h \|Q u\|^2_{L^2_r} + h \langle Y u,u\rangle_{L^2_r}.
\end{split}\end{equation}
We now estimate the right hand side term by term to obtain \eqref{e:bpc}. Since $P_B - \lambda$ is semiclassically elliptic on $\supp a$, by \eqref{ellipestrn} followed by \eqref{e:modelbound2} we have
\[
|\langle Au,u\rangle_{L^2_r}| \le C \|f\|_{L^2_r}^2 + Ch^2 \|u\|^2_{L^2_r} \le C\log^2(1/h) \|f\|_{L^2_r}^2.
\]
For any $\epsilon>0$ and $\chi_1 \in C_0^\infty(\R)$ with $\chi_1 = 1$ near $\supp f$ we have
\[
\frac 2 h\im\langle Q^*Qu,f\rangle_{L^2_r} \le \epsilon \| \chi_1 u\|^2_{L^2_r} + \frac C{h^2\epsilon} \|f\|^2_{L^2_r}. 
\]
By \eqref{e:bne0pc} and the elliptic estimate \eqref{ellipestrn}, if further $\inf \supp \chi_1 >-R_0$, then \eqref{e:chi0pc} gives
\[
\frac 2 h\im\langle Q^*Qu,f\rangle_{L^2_r} \le C \epsilon \| B u\|^2_{L^2_r} + \frac C{h^2\epsilon} \|f\|^2_{L^2_r}.
\]
Next we have, using $W_C \ge 0$ and  the fact that $h^{-1}[W_C, Q^*]Q$ has imaginary principal symbol, 
 followed by \eqref{e:modelbound2},
\[\begin{split}
- \frac 1 h\langle (W_C Q^*Q  + Q^*Q W_C)u,u\rangle_{L^2_r} &= - \frac 2 h \langle W_C Qu,Qu\rangle_{L^2_r} + \frac 2 h \re \langle [W_C, Q^*]Q u,u\rangle_{L^2_r}\\
& \le C h\|u\|^2_{L^2_r} \le C \frac{\log^2(1/h)}h\|f\|_{L^2_r}^2.
\end{split}\]
Finally we observe that $- 2 \im \lambda \|Q u\|^2_{L^2_r}/h \le 0$ since $\im \lambda \ge 0$, while \eqref{e:modelbound2} implies
\[
h \langle Y u,u\rangle_{L^2_r} \le C \frac{\log^2(1/h)}h\|f\|_{L^2_r}^2.
\]
This completes the estimation of \eqref{e:bpcterms} term by term, giving \eqref{e:bpc}.
\end{proof}

\begin{proof}[Proof of \eqref{e:modelpropa} for $\lambda_1 h \le \alpha \le \alpha_0$.] We begin this proof with the same rescaling to $\tilde r$ and $\tilde h$, and the same parametrix construction as for the proof of \eqref{e:modelbound2}  for $\lambda_1 h \le \alpha \le \alpha_0$ above, but with the additional requirement that
\[
-\widetilde R \le r_0/\log 2.
\]
Then if we put
\[
\widetilde \chi_+(\tilde r) = \chi_+(r), \qquad \widetilde \chi_-(\tilde r) = \chi_-(r),
\]
we have 
\[
\supp \widetilde \chi_+ \subset (r_0/\log(2\alpha_0/\alpha),\infty) \subset (r_0/\log2,\infty), \quad \supp \chi_E \subset (-\infty,-\widetilde R -1) ,
\]
 and hence
\begin{equation}\label{e:chipchie}
\widetilde\chi_+(\tilde r) \chi_E(\tilde r -1 )=0.
\end{equation}
Then, noting that \eqref{e:palphapar} implies
\[
(P(\alpha)-\lambda)^{-1} = G(\Id-A_B)(\Id + \Oh_{L^2_{\tilde r} \to L^2_{\tilde r} }(e^{-1/(C\tilde h)})),
\]
we use \eqref{e:chipchie} to write
\[
\widetilde\chi_+(\tilde r)(P(\alpha)-\lambda)^{-1} \widetilde\chi_-(\tilde r) = \widetilde\chi_+(\tilde r) R_B \widetilde\chi_-(\tilde r) + \Oh_{L^2_{\tilde r} \to H^2_{\tilde h, \tilde r} }(e^{-1/(C\tilde h)}). 
\]
Returning to the $r$ and $h$ variables, we see that  it suffices to show that
\begin{equation}\label{e:modelproppb}
\|\varphi(h D_{ r}) \chi_+( r) R_B\chi_-( r)\|_{L^2_{ r} \to H^2_{ h,  r}} = \Oh( h^\infty).
\end{equation}
The proof of \eqref{e:modelproppb} is almost the same as that of \eqref{e:modelprop0}. There are two differences.

The first difference is that as an escape function we use
\[
q = \varphi_r( r)\varphi_\rho( \rho), \qquad 
\re H_{p_B} q= 
2 \rho\varphi'_r( r)\varphi_\rho( \rho) -  V'_C( r) \varphi_r'( r)\varphi'_\rho( \rho),
\]
where $\varphi_r \in C_0^\infty(\R;[0,\infty))$ with $\supp \varphi_r \subset (r_0,\infty)$, $\varphi_r' \ge 0$ near $[r_0,0]$, $\varphi_r' > 0$ near $\supp \chi_+$. Take $\varphi_\rho \in C_0^\infty(\R;[0,\infty))$ with $\supp \varphi_\rho \subset (-\infty,0)$, $\varphi_\rho' \le 0$ near $[-2,0]$, $\varphi_\rho \ne 0$ near $\supp \varphi \cap [-2,0]$.  Impose further that $\sqrt\varphi_r, \sqrt\varphi_\rho \in C_0^\infty(\R)$, and that $\varphi'_r \ge c \varphi_r$ for $r \le 0$, where $c >0$ is chosen large enough that $\re H_{p_B} q \le -(2\Gamma +1)q$ on $\{r \le 0, \rho \ge -2\}$. 

The second difference is that the complex absorbing barrier $W_C$ produces a remainder term in the positive commutator estimate, analogous to the one in the proof of  \eqref{e:modelbound2b} for $\lambda_1 h \le \alpha \le \alpha_0$ above. The same argument removes the remainder term in this case.
\end{proof}

\section{Model operator in the funnel}\label{s:funnel}

Take $W_F \in C^\infty(\R;[0,1])$ nonincreasing with $W_F(r) = 0$ near $r \ge R_g$, $W_F(r) = 1$ near $r \le 0$, and let
\[
P_F = h^2D_r^2  + (1-W_F(r))e^{-2(r+\beta(r))} \Delta_{S_+} + h^2V(r) - 1 - iW_F(r),
\]
with notation as in \S\ref{spectrum}. 

\begin{prop} \label{p:modelboundf}
For every $\chi \in C_0^\infty(X)$, $E \in (0,1)$, there is $C_0>0$ such that for any $M>0$, there are $h_0,C>0$  such that the cutoff resolvent $\chi R_F(\lambda) \chi$ continues holomorphically from $\{\im \lambda>0\}$ to  $\{|\re \lambda| \le  E$, $-Mh\log(1/h) \le \im \lambda\}, \, h \in (0,h_0]$, where it satisfies
\begin{equation}\label{e:modelboundf}\left\|\chi R_F(\lambda) \chi\right\|_{L^2_\varphi(X) \to H^2_{\varphi,h} (X)} \le C \begin{cases}
h^{-1} + |\lambda|, \qquad & \im \lambda > 0 \\
h^{-1}e^{C_0 |\im \lambda|/h}, \qquad &\im \lambda \le 0,
\end{cases}. \end{equation}
\end{prop}

\begin{prop}\label{p:modelpropf}
Let $r_0 > R_g$, $\chi_- \in C_0^\infty((-\infty,r_0))$, $\chi_+ \in C_0^\infty((r_0,\infty))$, $\varphi \in C^\infty(\R)$ supported in $(0,\infty)$ and bounded with all derivatives, $E \in (0,1)$, $\Gamma>0$ be given. Then there exists $h_0>0$ such that
\begin{equation}\label{e:modelpropf}
\left\|\chi_+(r)R_F(\lambda)\chi_-(r)\varphi(hD_r)\right\|_{L^2_\varphi(X) \to H^2_{\varphi,h}(X)} = \Oh(h^\infty),
\end{equation}
 for $|\re \lambda| \le E, \, - \Gamma h \le \im \lambda \le h^{-N}$, $h \in (0,h_0]$.
\end{prop}

To prove these propositions we separate variables over the eigenspaces of $\Delta_{S_+}$, writing $P_F = \bigoplus_{m=0}^\infty  h^2D_r^2 +  (1-W_F(r))(h\lambda_m)^2 e^{-2(r + \beta(r))} + h^2 V(r) - 1 - iW_F(r),$ 
where $0 = \lambda_0 < \lambda_1  \le \cdots$ are square roots of the eigenvalues of $\Delta_{S_+}$. It suffices to prove \eqref{e:modelboundf}, \eqref{e:modelpropf} with $P_F$ replaced by $P(\alpha)$,  with estimates uniform in $\alpha \ge 0$, where
\[
P(\alpha) = h^2D_r^2 +  (1-W_F(r))\alpha^2e^{-2(r + \beta(r))} + h^2 V(r) - 1 - iW_F(r).
\]

Next we use a variant of the method of complex scaling presented in the proof of Lemma \ref{l:p0}, but with contours $\gamma$ depending on $\alpha$ in such a way as to give estimates uniform in $\alpha$; the $\alpha$-dependence is needed because the term $\alpha^2(1-W_F(r))e^{-2(r + \beta(r))}$, although exponentially decaying, is not uniformly exponentially decaying as $\alpha \to \infty$. Such contours were first used in \cite[\S 4]{z}; here we present a simplified approach based on that in \cite[\S 5.2]{Datchev:Thesis}.

Fix $R>R_g$ sufficiently large that
\[
\supp \chi \cup \supp \chi_+ \cup \supp \chi_- \subset(-\infty,R).
\]
and that
\begin{equation}\label{e:betanegf}
\re z \ge R, \ 0 \le \arg z \le \theta_0 \Longrightarrow |\im \beta(z)| \le |\im z|/2,
\end{equation}
where $\theta_0$ is as in \S\ref{s:assumptions}.
Let $\gamma = \gamma_\alpha(r)$ be real-valued, smooth in $r$ with $\gamma'(r)\ge 0$ for all $r$, and obey $\gamma(r) = 0$ for $r \le R$ (here and below $\gamma' = \partial_r \gamma$). Suppose $\gamma'' \in C_0^\infty(\R)$ for each $\alpha$, but not necessarily uniformly in $\alpha$. Now put
\[\begin{split}
P_\gamma(\alpha) = \frac{h^2D_r^2}{(1 + i\gamma'(r))^2} - h \frac{\gamma''(r) h D_r}{(1 + i\gamma'(r))^3}+ \alpha^2(1-W_F(r))e^{-2(r + i \gamma(r) + \beta(r + i \gamma(r)))} \\ + h^2 V(r + i \gamma(r)) - 1 - iW_F(r).
\end{split}\]
If we define the differential operator with complex coefficients
\[\widetilde P(\alpha) = h^2D_z^2 + \alpha^2(1-W_F(z))e^{-2(z + \beta(z))} + h^2 V(z) - 1 - iW_F(z), 
\]
then we have
\[
P(\alpha) = \widetilde P(\alpha)|_{\{z=r: r \in \R\}}, \qquad P_\gamma(\alpha) = \widetilde P(\alpha)|_{\{z=r + i \gamma(r): r \in \R\}}.
\]
If $\chi_0 \in C^\infty(\R)$ has $\supp \chi_0 \cap \supp \gamma = \varnothing$, then
\[
\chi_0(P(\alpha) - \lambda)^{-1} \chi_0 = \chi_0(P_\gamma(\alpha) - \lambda)^{-1} \chi_0, \qquad \im \lambda > 0,
\]
by an argument almost identical to that used to prove \eqref{e:p0agree}; the only difference is we construct WKB solutions which are exponentially growing and decaying as $\re z \to +\infty$ rather than $-\infty$, and we take $
f(z) = (\alpha^2 e^{-2(z + \beta(z))} + h^2 V(z) -1 - \lambda)/h^2.$

Consequently to prove \eqref{e:modelboundf} and \eqref{e:modelpropf}, it is enough to show that
\begin{equation}\label{e:modelboundfg}\left\|(P_\gamma(\alpha) - \lambda)^{-1} \right\|_{L^2(\R) \to H^2_h (\R)} \le C  e^{C_0 |\im \lambda|  /h}, \end{equation}
and
\begin{equation}\label{e:modelpropfg}
\left\|\chi_+(r)(P_\gamma(\alpha) - \lambda)^{-1}\chi_-(r) \varphi(hD_r)\right\|_{L^2(\R) \to H^2_h(\R)} = \Oh(h^\infty),
\end{equation}
for a suitably chosen $\gamma$, with estimates uniform in $\alpha \ge 0$.

 Fix $R_->R$ such that 
\begin{equation}\label{e:imbeta}
|\im \beta(z)| \le \im z/2
\end{equation}
 for $\re z \ge R_-$, $0 \le \arg z \le \theta_0$, with $\theta_0$ as in \S\ref{s:assumptions}.
Take $\alpha_0 > 0$ such that 
\begin{equation}\label{e:a0f}
\alpha_0^2 e^{-2(R+1)} e^{-2\max|\re\beta|} = 8,
\end{equation}
where $\max |\re \beta|$ is taken over $\R \cup \{|z| >R_g, \ 0 \le \arg z \le \theta_0\}$. We consider the cases $\alpha \le \alpha_0$ and $\alpha \ge \alpha_0$ separately.

\begin{proof}[Proof of \eqref{e:modelboundfg} for $0 \le \alpha \le \alpha_0$] Fix 
\[E_0 \in (E,1), \qquad \eps = 10Mh\log(1/h).\]

We use the same complex scaling as in the proof of Lemma \ref{l:p0}.
In this range  $\gamma$ is independent of $\alpha$ and we put $\gamma = \delta \gamma_-$, where $0<\delta \ll 1$ will be specified later, and we require $\gamma_-(r) =0$ for $r \le R_-$, $\gamma_-'(r) \ge 0$ for all $r$, and $\gamma_-'(r) = \tan \theta_0$ for $r \ge R_-+1$.

The semiclassical principal symbol of $P_\gamma(\alpha)$ is
\[\begin{split}
p_\gamma(\alpha) &= \frac{\rho^2}{(1 + i\gamma'(r))^2} + \alpha^2(1-W_F(r))e^{-2(r + i \gamma(r) + \beta(r + i \gamma(r)))} - 1 - iW_F(r) \\
&=  \rho^2 + \alpha^2(1-W_F(r))e^{-2(r  + \beta(r ))} - 1 - iW_F(r) + \Oh(\delta).
\end{split},\]
where the implicit constant in $\Oh$ is uniform in compact subsets of $T^*\R$. Moreover,
\[
\re p_\gamma(\alpha) + 1 \ge \rho^2 - \Oh(\delta),
\]
and, using \eqref{e:imbeta},
\begin{equation}\label{e:ima0}
\begin{split}
\im p_\gamma(\alpha) 
& \le -\alpha^2(1-W_F(r))e^{-2(r + \re \beta(r+i\gamma(r))}\sin(2(\gamma(r) + \im \beta(r+ i \gamma(r)))\\
& \le  -\alpha^2(1-W_F(r))e^{-2(r + \re \beta(r+i\gamma(r))}\sin\gamma(r)\\
&=  -\alpha^2(1-W_F(r))e^{-2(r + \re \beta(r + i \gamma(r))}\gamma(r) (1 + \Oh(\delta^2)), 
\end{split}
\end{equation}
again uniformly on compact subsets of $T^*\R$.
Take $q \in C_0^\infty(T^*\R)$ such that on $\{0 \le r \le R_-+1, \ |\rho| \le 2\}$ we have
\[\begin{split}
q &= - C_q (r+1) \rho,\\
\frac{\re H_{p_\gamma} q}{C_q} 
&= -2 \rho^2 -(W'_F(r) + 2(1 + \beta'(r))(r+1)\alpha^2e^{-2(r  + \beta(r))} + \Oh(\delta)
\\&\le - (\re p_\gamma + 1) \le - \rho^2 + \Oh(\delta),
\end{split}\]
where $C_q>0$ will be specified later, and provided $\delta$ is sufficiently small. Let $Q = \Op(q)$ and put
\[
P_{\gamma,\eps}(\alpha) = e^{\eps Q/h} P_\gamma(\alpha) e^{- \eps Q/h} = P_\gamma(\alpha)  - \eps[P_\gamma(\alpha),Q/h] + \eps^2R,
\]
where $R \in \Psi^{-\infty}(\R)$ (see \eqref{expexp}). As in the proof of Lemma \ref{l:p0}, \eqref{e:modelboundfg} follows from
\begin{equation}\label{e:modelboundfeps0}
\left\|(P_{\gamma,\eps}(\alpha) - E')^{-1} \right\|_{L^2(\R) \to H^2_h (\R)} \le 5/\eps,
\end{equation}
for $E' \in [-E_0,E_0]$. 

The proof of \eqref{e:modelboundfeps0} combines  elements of the proofs of \eqref{e:modelbound0eps0} and \eqref{e:modelbouncd0eps}. Let $\phi \in C_0^\infty(T^*\R)$ be identically $1$ near $\{0 \le r \le R_- + 1,\ | \rho|\le2,\ |\re p_\gamma| \le E_0\}$ and be supported such that $\re H_{p_\gamma}q <0$ on $\supp \phi$. Let $\Phi$ be the quantization of $\phi$. For $\delta$ small enough, and $h$ (and hence $\eps$) small enough depending on $\delta$,  we have $|p_{\gamma,\eps} - E'| \ge \delta (1 +  \rho^2)/C$ on $\supp (1-\phi)$, uniformly in $E' \in [-E_0,E_0]$, in $\alpha \le \alpha_0$ and in $h$, where $p_{\gamma,\eps}(\alpha)$ is the semiclassical principal symbol of $P_{\gamma,\eps}(\alpha)$. Hence, by the semiclassical elliptic estimate \eqref{ellipestrn},
\[
\|(\Id - \Phi) u\|_{H^2_{h}(\R)} \le C \delta^{-1}\|(P_{\gamma,\eps} - E')(\Id - \Phi) u\|_{L^2(\R)} + \Oh(h^\infty)\|u\|_{H^{-N}_{ h}(\R)}.
\]
Using \eqref{e:ima0} and $\supp\phi \subset \{\re H_{p_c}q <0\}$, fix $C_q$ large enough that on $\supp\phi$  we have
\[
\im p_{\gamma,\eps} =  \im p_\gamma + \eps \re H_{p_c} q
\le  -\alpha^2(1-W_F)e^{-2(r + \re \beta)} \gamma (1 + \Oh(\delta^2)) + \eps \re H_{p_c} q \le -\eps.
\]
Then, using the sharp G\aa rding inequality \eqref{gardingrn}, we have, for $h$ sufficiently small,
\[\begin{split}
\|\Phi u\|_{L^2(\R)} \|(P_{C,\eps} - E')\Phi u\|_{L^2(\R)}  &\ge - \la\im (P_{C,\eps} - E')\Phi u, \Phi u \ra_{L^2(\R)} \\
& \ge \eps \|\Phi u\|_{L^2(\R)}^2 - C h \|u\|^2_{L^2(\R)} .
\end{split}\]
This implies  \eqref{e:modelboundfeps0}  just as in the proofs of \eqref{e:modelbound0eps0} and \eqref{e:modelbouncd0eps}. 
\end{proof}

\begin{proof}[Proof of \eqref{e:modelboundfg} for $\alpha \ge \alpha_0$]  Define contours $\gamma = \gamma_\alpha(r)$ as follows. Take $R_\alpha$ such that  
\begin{equation}\label{e:ralphadef}
\alpha^2 e^{-2R_\alpha} e^{2 \max |\re \beta|} = \min\{1/4,(\tan \theta_0)/2\},
\end{equation}
where $\max |\re \beta|$ is taken over $\R \cup \{|z| >R_g, \ 0 \le \arg z \le \theta_0\}$.
Note that  $R_\alpha > R+1$ by \eqref{e:a0f}. Take $\gamma$  smooth and supported in $(R,\infty)$, with $0 \le \gamma'(r) \le 1/2$, and such that
\[\begin{split}
\gamma(r) \le \pi/9, \quad &r \le R+1,\\
\pi/18 \le \gamma(r) \le \pi/6, \quad &R+1 \le r \le R_\alpha,\\
\gamma'(r) = \min\{1/2,\tan \theta_0\}, \quad & r \ge R_\alpha.
\end{split}\]
We prove that
\begin{equation}\label{e:pgaellip}
|p_\gamma(\alpha) - E'| \ge (1+\rho^2)/C,
\end{equation}
uniformly for $-E \le E' \le E$ and $\alpha \ge \alpha_0$, by considering each range of $r$ individually. By \eqref{ellipestrn} this implies \eqref{e:modelboundfg} for $\alpha \ge \alpha_0$.

\begin{enumerate}
\item For $r \le R+1$ we have
\begin{equation}\label{e:repgar}\begin{split}
\re p_\gamma(\alpha) + 1 &= \frac{\rho^2(1-\gamma'(r)^2)}{|1 + i\gamma'(r)|^4} + \alpha^2(1-W_F(r))\re e^{-2(r + i \gamma(r) + \beta(r + i \gamma(r)))}  \\
&\ge\frac 13 \rho^2 +  \alpha^2(1-W_F(r))e^{-2(r+\re \beta(r+i\gamma(r)))} \cos (3\gamma(r)) \\
&\ge \frac 13 \rho^2 + 4(1-W_F(r)) ,
\end{split}\end{equation}
where for the first inequality we used $\gamma'\le 1/2$ and \eqref{e:imbeta}, and for the second \eqref{e:a0f} and $\gamma \le \pi/9$.
Since $\im p_\gamma = -W_F$ whenever $W_F \ne 0$, this gives \eqref{e:pgaellip} for $r \le R+1$.

\item For $R+1 \le r \le R_\alpha$ we have $\re p_\gamma(\alpha) \ge \frac 13 \rho^2 - 1$ by the same argument as in \eqref{e:repgar}. This gives \eqref{e:pgaellip} for $R+1 \le r \le R_\alpha$ once we note that \eqref{e:imbeta} and \eqref{e:ralphadef} imply
\[\begin{split}
-\im p_\gamma(\alpha) &=  \frac{2\rho^2\gamma'(r)}{|1 + i\gamma'(r)|^4} - \alpha^2\im e^{-2(r + i \gamma(r) + \beta(r + i \gamma(r)))} \\&\ge e^{-2\max|\re\beta|}\sin(\pi/18)\min\{1/2,(\tan \theta_0)/2\}.
\end{split}\]

\item For $r \ge R_\alpha$, note that $\alpha^2 |e^{-2(r + i \gamma(r) + \beta(r + i \gamma(r)))}| \le \gamma'(r)$. We again deduce \eqref{e:pgaellip} by considering two ranges of $\rho$ individually. When $\rho^2/|1 +i\gamma'(r)|^4 \le 1/2$ we have
\[\begin{split}
\re p_\gamma(\alpha) &= \frac{\rho^2(1-\gamma'(r)^2)}{|1 + i\gamma'(r)|^4} + \alpha^2\re e^{-2(r + i \gamma(r) + \beta(r + i \gamma(r)))} - 1 \\
&\le 1/2 + 1/4- 1 = -1/4.
\end{split}\]
When $\rho^2/|1 +i\gamma'(r)|^4  \ge 1/2$ we have
\[\begin{split}
\im p_\gamma(\alpha) &=  \frac{-2\rho^2\gamma'(r)}{|1 + i\gamma'(r)|^4} + \alpha^2\im e^{-2(r + i \gamma(r) + \beta(r + i \gamma(r)))}\\
&\le  \frac{-2\rho^2\gamma'(r)}{|1 + i\gamma'(r)|^4} + \frac{\gamma'(r)}2 \le -3\gamma'(r)/2= -\min\{3/4,3(\tan \theta_0)/2\}.
\end{split}\] 
\end{enumerate}
\end{proof}

For $\alpha \ge \alpha_0$, \eqref{e:modelpropfg} follows from an Agmon estimate just as in the proof of  \eqref{e:modelpropa} for $\alpha \ge \alpha_0$ above. For $\alpha \le \alpha_0$, \eqref{e:modelpropfg} follows from the same positive commutator argument as was used for the proof of \eqref{e:modelproppb}.

\section{Applications}\label{s:applications}
In this section we use the notation
\[
\|u\|_s = \|(1 + \Delta)^{s/2} u\|_{L^2(X)},\ \|A\|_{s \to s'} = \sup_{\|u\|_s = 1} \|A u\|_{s'}, \qquad s,s' \in \R.
\]
We begin by using \eqref{logreg} to deduce polynomial bounds on the resolvent between Sobolev spaces. If $\chi, \widetilde \chi \in C_0^\infty(X)$ have $\widetilde \chi \chi = \chi$, then for any $s \in \R$, we have
\[\begin{split}
\|\Delta \chi u\|_s &\le C(\|\widetilde \chi u\|_s + \|\widetilde\chi \Delta u\|_s).
\end{split}\]
Hence, for any $s,s' \in \R$, we have, if $R_\chi(\sigma) = \chi (\Delta - n^2/4 - \sigma^2)^{-1} \chi$,
\[\begin{split}
\|R_\chi(\sigma)\|_{s \to s} &\le C \|R_{\widetilde \chi}(\sigma)\|_{s' \to s'},\\
\|R_\chi(\sigma)\|_{s \to s'+2} &\le C(1 + |\sigma|^2) \left(\|R_{\widetilde \chi}(\sigma)\|_{s \to s} + \|R_{\widetilde \chi}(\sigma)\|_{s \to s'}\right), \\
\|R_\chi(\sigma)\|_{s \to s'} &\le C(1 + |\sigma|^2)^{-1} \left(\|R_{\widetilde \chi}(\sigma)\|_{s \to s'+2} + \|R_{\widetilde \chi}(\sigma)\|_{s \to s'}\right). 
\end{split}\]
Consequently, for any $\chi \in C_0^\infty(X)$, there is $M_0>0$ such that for any $M_1>0$, $s\in \R$, $s' \le s+2$ there is $M_2>0$ such that
\begin{equation}\label{e:ressob}
\|R_\chi(\sigma)\|_{s \to s'} \le M_2 |\sigma|^{M_0|\im \sigma| + s' - s-1},
\end{equation}
when $|\re \sigma| \ge M_2$, $\im \sigma \ge -M_1$.

\subsection{Local smoothing}
By the self-adjoint functional calculus of $\Delta$, the Schr\"odinger propagator is unitary on all Sobolev spaces: for any $s,t \in \R$, if $u \in H^s(X)$,
\[
\|e^{-it\Delta}u\|_s = \|u\|_s.
\]
The Kato local smoothing effect says that if we localize in space and average in time, then Sobolev regularity improves by half a derivative:
for any $\chi\in C_0^\infty(X)$, $T>0$, $s \in \R$ there is $C>0$ such that if $u \in H^s(X)$,
\begin{equation}\label{e:locsmo}
\int_0^T \left\|\chi e^{-it\Delta} u\right\|^2_{s+1/2}dt \le C \|u\|^2_s.
\end{equation}
This follows by a $TT^*$ argument from \eqref{e:ressob} applied with $\im \sigma = s = 0$, $s' = 1$ (see e.g. \cite[p 424]{bur:smoothing}); note that in this case the bound is uniform as $\sigma \to \pm \infty$.

\subsection{Resonant wave expansions}\label{s:wave} Suppose $\chi (\Delta - n^2/4 - \sigma^2)^{-1} \chi$ is meromorphic for $\sigma \in\C$. For example we may take $(X,g)$ as in \S\ref{infmany}. More generally, if the funnel end is evenly asymptotically hyperbolic as in \cite[Definition 1.2]{g} then this  follows as in the proof of Theorem 1.1 in \cite[p 747]{sz2}, but in the interest of brevity we do not pursue this here.

Then \eqref{e:ressob} implies that, when the initial data is compactly supported,  solutions to the wave equation $(\partial_t^2 + \Delta - n^2/4)u = 0$ can be expanded into a superposition of eigenstates and resonant states, with a remainder which decays exponentially on compact sets:

Let $s \in \R$, $\chi \in C_0^\infty(X)$, $f \in H^{s+1}(X)$, $g \in H^{s}(X)$, $\chi f = f$, $\chi g = g$. For any $M_1>0$,
\begin{equation}\label{e:sobthresh}
s' < s -M_0M_1,
\end{equation}
 there are $C,T>0$ such that if $t \ge T$, $H= \sqrt{\Delta - n^2/4}$, then 
 \[
 \left\|\chi \left(\cos(tH) f + \frac{\sin (tH)}{H} g-   \sum_{\im \sigma_j > -M_1} \sum_{m=1}^{M(\sigma_j)} e^{-i\sigma_j t} t^{m-1} w_{j,m} \right) \right\|_{s'}
  \le C e^{-M_1 t},
\]
where the sum is taken over poles of $R_\chi(\sigma)$ (and is finite by the Theorem),  $M(\sigma_j)$ is the rank of the residue of the pole at $\sigma_j$, and each $w_{j,m}$ is a linear combination of the projections of $f$ and $g$ onto the $m$-th eigenstate or resonant state at $\sigma_j$.
This follows from \eqref{e:ressob} by an argument of \cite{lp,v}; see also \cite[Theorem 3.3]{tz2}  or \cite[Corollary 6.1]{Datchev-Vasy:Gluing}.

\textbf{Remark}. The local smoothing estimate \eqref{e:locsmo} is lossless in the sense that the  result is the same if $(X,g)$ is nontrapping and asymptotically Euclidean or hyperbolic (see \cite[(1.6)]{cpv} for a general result). This is because the resolvent estimates \eqref{logreg} and \eqref{e:eucbetter} agree when $\im \sigma = 0$. The resonant wave expansion exhibits a loss in the Sobolev spaces in which the remainder is controlled: the improvement from \eqref{logreg} to \eqref{e:eucbetter} for $\im \sigma < 0$ means that, when \eqref{e:eucbetter} holds, we can replace \eqref{e:sobthresh} with $s' < s$.

\section{Lower bounds}\label{s:low}
In this section we prove that, in the setting of an exact quotient, the holomorphic continuation of the resolvent grows  polynomially. As in \cite[\S 5.3]{b}, we use the fact that in this case integral kernel of the resolvent can be written in terms of modified Bessel functions.

\begin{prop}\label{p:bessel}
Let $(X,g)$ be given by 
\[
X = \R \times S, \qquad g = dr^2 + e^{2r}dS,
\]
 where $(S,dS)$ is a compact Riemannian manifold without boundary of dimension $n$. Then for any $\chi \in C_0^\infty(X)$ which is not identically $0$, the cutoff resolvent $\chi(\Delta - n^2/4 - \sigma^2)^{-1}\chi $ continues holomorphically from $\{\im \sigma > 0\}$ to $\C \setminus 0$, with a simple pole of rank $1$ at $\sigma = 0$. 
 
 Moreover, if $\chi \ne 0$ in a neighborhood of $0$, for any $\eps > 0$ there exists $ C>0$ such that
\[
\|\chi(\Delta - n^2/4 - \sigma^2)^{-1}\chi\| \ge e^{-C|\im\sigma|} |\sigma|^{2|\im\sigma|-1}/C,
\]
when $\im \sigma \le -\eps$, $|\re \sigma| \ge C$, $|\im\sigma|\le |\re\sigma|/\eps$.
\end{prop}

\begin{proof}
As in \S\ref{spectrum} a conjugation and separation of variables reduce this to the study of the following family of ordinary differential operators
\[
P_m = D_r^2 + \lambda_m^2 e^{-2r},
\]
where $0=\lambda_0 < \lambda_1 \le \lambda_2 \le \cdots$ are square roots of the eigenvalues of $\Delta$. We will show that $\chi(P_m - \sigma^2)^{-1}\chi$ is entire in $\sigma$ for $m > 0$, and that it is holomorphic in $\C \setminus 0$ with a simple pole of rank $1$ at $\sigma = 0$ for $m = 0$. We will further show that
\[
\|\chi (P_1 - \sigma^2)^{-1} \chi\| \ge e^{-C|\im\sigma|} |\sigma|^{2|\im\sigma|-1}/C,
\]
when $\im \sigma \le -\eps$, $|\re \sigma| \ge C$, $|\im\sigma|\le |\re\sigma|/\eps$.

We write the integral kernel of the resolvent of each $P_m$ using the following formula (see for example \cite[(1.25)]{tz}):
\begin{equation}\label{e:resolventformula}
R_m(r,r') = -\psi_1(\max\{r,r'\})\psi_2(\min\{r,r'\})/W(\psi_1,\psi_2),
\end{equation}
where $\psi_1$ and $\psi_2$ are linearly independent solutions to $(P_m -\sigma^2)u=0$ and $W(\psi_1,\psi_2)$ is their Wronskian. 

If $m = 0$ we take $\psi_1(r)=e^{i r \sigma}$  and $\psi_2(r) = e^{-i r \sigma}$ (this is the only choice for which the resolvent maps $L^2 \to L^2$ for $\im \sigma > 0$), so that $W(\psi_1,\psi_2)=2i\sigma$. Now the asserted continuation is immediate from the formula \eqref{e:resolventformula}. 

To study $m > 0$ we use, as in \cite[\S5.3]{b}, the Bessel functions
\begin{equation}\label{e:besselres}\psi_1(r) =I_\nu\left(\lambda_m e^{-r} \right), \qquad \psi_2(r) = K_\nu\left(\lambda_m e^{-r} \right), \qquad \nu = -i\sigma.\end{equation}
We recall the definitions:
\begin{equation}\label{e:idef}I_\nu(z) = \frac{z^\nu}{2^\nu} \sum_{k=0}^\infty \frac{(z/2)^{2k}}{k!\Gamma(\nu+k+1)},\end{equation}
\begin{equation}\label{e:kdef}K_\nu(z) = \frac \pi{2\sin(\pi\nu)} \left(I_{-\nu}(z) - I_\nu(z)\right).\end{equation}

This pair solves the desired equation (see for example \cite[Chapter 7, (8.01)]{Olver:Asymptotics}) and has $W=1$ (see for example \cite[Chapter 7, (8.07)]{Olver:Asymptotics}). When $\im \sigma > 0$, we have $\re \nu > 0$ and this resolvent maps $L^2 \to L^2$ thanks to the asymptotic
\begin{equation}\label{e:iseries} I_\nu(z) = \frac{z^\nu}{2^\nu \Gamma(\nu+1)}\left(1 + \Oh\left(\frac{z^2}{\nu}\right)\right),\end{equation} 
which is a consequence of \eqref{e:idef}, and thanks to the fact that $K_\nu(z) \sim e^{-z}\sqrt{\pi/2z}$ as $z \to \infty$
(see for example \cite[Chapter 7, (8.04)]{Olver:Asymptotics}). Because $I$ and $K$ are entire in $\nu$, we have the desired homolorphic continuation of the resolvent for all $m > 0$.

To estimate the resolvent we use \eqref{e:kdef} and \eqref{e:iseries} to write
\[\begin{split}
I_\nu(z')K_\nu(z) &= \frac \pi {2\sin(\pi\nu)}I_\nu(z')(I_{-\nu}(z) - I_\nu(z)) \\
& =\frac \pi{\sin(\pi\nu)\Gamma(\nu+1)}  \frac {{z'}^\nu}{2^{\nu+1}}\left( \frac{z^{-\nu}}{2^{-\nu} \Gamma(-\nu+1)} -  \frac{z^\nu}{2^\nu \Gamma(\nu+1)}\right)\left(1 + \Oh\left(\frac{z^2 + {z'}^2}{\nu}\right)\right).
\end{split}\]
Using Euler's reflection formula for the Gamma function (see for example \cite[Chapter 2, (1.07)]{Olver:Asymptotics}),
\[\frac \pi{\sin(\pi\nu)\Gamma(\nu+1)} =  -\Gamma(-\nu) = \frac{\Gamma(-\nu+1)}\nu,\]
it follows that
\begin{equation}\label{e:ikprod}\begin{split}
I_\nu(z')K_\nu(z) &= \frac {{z'}^\nu}{2^{\nu+1}\nu}\left( \frac{z^{-\nu}}{2^{-\nu}} -  \frac{z^\nu\Gamma(-\nu+1)}{2^\nu \Gamma(\nu+1)}\right)\left(1 + \Oh\left(\frac{z^2 + {z'}^2}{\nu}\right)\right) \\
&= \frac {{z'}^\nu}{2^{\nu+1}\nu}\left( \frac{z^{-\nu}}{2^{-\nu}} +  \frac{\nu z^\nu\sin(\pi\nu)\Gamma(-\nu)^2}{2^\nu \pi}\right)\left(1 + \Oh\left(\frac{z^2 + {z'}^2}{\nu}\right)\right).
\end{split}\end{equation}
Using Stirling's formula (see for example \cite[Chapter 8, (4.04)]{Olver:Asymptotics})
\[
\Gamma(-\nu) = e^\nu (-\nu)^{-\nu}\sqrt{-2\pi/\nu} (1 + \Oh(\nu^{-1})),
\]
for $\arg (- \nu)$ varying in a compact subset of $(-\pi,\pi)$ and with the branch of $ (-\nu)^{-\nu}$ taken to be real and positive when $-\nu$ is, we write
\[\begin{split}
|\nu\sin(\pi\nu)\Gamma(-\nu)^2| 
& =  \pi e^{\pi|\im\nu|} e^{2\re\nu} |\nu|^{-2\re \nu} e^{2\im\nu \arg(-\nu)} (1+ \Oh(|\im \nu|^{-1})),\\
&= \pi  e^{2\re\nu} |\nu|^{-2\re \nu} e^{-2\im\nu \arctan\frac{\re\nu}{\im\nu}}(1+ \Oh(|\im \nu|^{-1}))\\
&=  \pi  |\nu|^{-2\re \nu} e^{-\frac 23 (\re \nu)^3/(\im \nu)^2}(1+ \Oh(|\re\nu|^5|\im\nu|^{-4} +  |\im \nu|^{-1})),
\end{split}\]
for $\arg \nu$ varying in a compact subset of $(0,2\pi)$.

To bound the resolvent from below we apply it to the characteristic function of an interval: let $ a > 0 $ and put
\[\begin{split}
u(r) &= -\int_0^a R_1(r,r')dr' = K_\nu(\lambda_1 e^{-r}) \int_0^a I_\nu(\lambda_1 e^{-r'}) dr',
\end{split}\]
where the last equality holds only for $r \le 0$.
Then if $\chi \in C^\infty(\R)$ is identically $1$ on $[-a,a]$ we have
\[\begin{split}
\|\chi (P_1- \sigma^2)^{-1} \chi\|^2 &\ge \frac 1 a\int_{-a}^{a} |u(r)|^2 dr \ge \frac 1 a\int_{-a}^0 \left |K_\nu(\lambda_1 e^{-r}) \int_0^a I_\nu(\lambda_1 e^{-r'}) dr'\right|^2 dr\\
&= \frac 1 {a} \left| \int_0^a I_\nu(\lambda_1 e^{-r'}) dr'\right|^2\int_{-a}^0 \left |K_\nu(\lambda_1 e^{-r})\right|^2 dr.
\end{split}\]
Using \eqref{e:ikprod} we obtain
\[
\|\chi (P_1- \sigma^2)^{-1} \chi\|^2 \ge \frac 1 {4a} \left| \int_{-a}^a \frac {(\lambda_1 e^{-r'})^\nu}{2^{\nu}\nu} dr'\right|^2\int_{-2a}^{-a} \left |\frac{(\lambda_1e^{-r})^{-\nu}}{2^{-\nu}} +  \frac{\nu (\lambda_1e^{-r})^\nu\sin(\pi\nu)\Gamma(-\nu)^2}{2^\nu \pi}\right|^2 dr,
\]
provided $ |\nu| ^{-1}  \le \lambda^{-2}_1 e^{-2a}/c_0$ for a suitably large absolute constant $c_0$.
However,
\[\begin{split}
&\left| \int_{-a}^a \frac {(\lambda_1 e^{-r'})^\nu}{2^{\nu+1}\nu} dr'\right| = \frac{\lambda_1^{\re \nu}}{2^{\re\nu + 1}|\nu|^2}|e^{a\nu}-e^{-a\nu}|\ge \\&\frac{\lambda_1^{\re \nu}}{2^{\re\nu + 1}|\nu|^2}\left(e^{a |\re \nu|} - e^{-a|\re\nu|}\right) \ge e^{-C|\re \nu|}/(C|\nu|^2).
\end{split}\]
Then define $f(\nu)$ and $g(\nu)$ by
\[\begin{split}
\left |\frac{(\lambda_1e^{-r})^{-\nu}}{2^{-\nu}} +  \frac{\nu (\lambda_1e^{-r})^\nu\sin(\pi\nu)\Gamma(-\nu)^2}{2^\nu \pi}\right| &\ge  \frac 12 |\nu|^{-2\re \nu} e^{-\frac 23 \frac{(\re \nu)^3}{(\im \nu)^2}}  \frac{(\lambda_1e^{-r})^{\re\nu}}{2^{\re\nu}} - \frac{2^{\re\nu}}{(\lambda_1e^{-r})^{\re\nu}} \\
&= f(\nu) g(\nu) e^{-\re \nu r} - e^{\re \nu r}/g(\nu).
\end{split}\]
So, provided $\re \nu \le0$,
\[\begin{split}
\int_{-2a}^a \left |\frac{(\lambda_1e^{-r})^{-\nu}}{2^{-\nu}} +  \frac{\nu (\lambda_1e^{-r})^\nu\sin(\pi\nu)\Gamma(-\nu)^2}{2^\nu \pi}\right|^2dr &\ge \int_{-2a}^{-a} \left(f^2g^2 e^{-2\re \nu r} - 2f\right) dr \\
&\ge a(f^2g^2e^{-4|\re \nu| a} -2f).
\end{split}\]
Then if additionally $2 \le fg^2e^{-4|\re \nu| a}/2$ (it suffices to require $\re \nu \le - \eps$ and then $|\nu|$ sufficiently large depending on $\eps$), we have
\[
\int_{-2a}^a \left |\frac{(\lambda_1e^{-r})^{-\nu}}{2^{-\nu}} +  \frac{\nu (\lambda_1e^{-r})^\nu\sin(\pi\nu)\Gamma(-\nu)^2}{2^\nu \pi}\right|^2dr \ge a  f^2g^2e^{-4|\re \nu| a}/2,
\]
so that
\[
\|\chi (P_1- \sigma^2)^{-1} \chi\|^2 \ge \frac {e^{-C|\re \nu|}}{C|\nu|^2}|\nu|^{4|\re \nu|}. 
\]
\end{proof}

\section*{Appendix. The curvature of a warped product}

The result of this calculation is used in the examples in \S\ref{s:examples}, and  although it is well known, we include the details for the convenience of the reader. 
For this section only, let $(S,\tilde g)$ be a compact Riemannian manifold, and let $X = \R \times S$ have the metric
\[g = dr^2 + f(r)^2 \tilde g,\]
where $f \in C^\infty(\R;(0,\infty))$. Let $p \in X$, let $P$ be a two-dimensional subspace of $T_pX$, and let $K(P)$ be the sectional curvature of $P$ with respect to $g$. We will show that if $\D_r \in P$, then
\[K(P) = - f''(r)/f(r),\]
while if $P \subset T_pS$ and $\widetilde K(P)$ is the sectional curvature of $P$ with respect to $\tilde g$, then
\[K(P) = (\widetilde K(P) - f'(r)^2)/f(r)^2.\]

We work in coordinates $(x^0,\dots,x^n)=(r,x^1,\dots,x^n)$, and write
\[g = g_{\alpha\beta}dx^\alpha dx^\beta = dr^2 + g_{ij}dx^idx^j = dr^2 + f(r)^2 \tilde g_{ij}dx^idx^j,\]
using the Einstein summation convention. We use Greek letters for indices which include $0$, that is indices which include $r$, and Latin letters for indices which do not. Then
\[\D_\alpha g_{r\alpha} = 0,\qquad \D_rg_{jk} = 2f^{-1}f'g_{jk}, \qquad \D_ig_{jk} = f^2\D_i\tilde g_{jk}.\]
We write $\Gamma$ for the Christoffel symbols of $g$, and $\widetilde\Gamma$ for those of $\tilde g$. These are given by
\[{\Gamma^r}_{r\alpha} = {\Gamma^\alpha}_{rr} = 0, \qquad {\Gamma^r}_{jk} = -f^{-1}f'g_{jk}, \qquad {\Gamma^i}_{jr} = f^{-1}f' \delta^i_j, \qquad {\Gamma^i}_{jk} = {\widetilde \Gamma^i}_{jk}.\]
Let $R$ be the Riemann curvature tensor of $g$:
\[{R_{\alpha\beta\gamma}}^\delta = \D_\alpha {\Gamma^\delta}_{\beta\gamma} + {\Gamma^\eps}_{\beta\gamma}{\Gamma^\delta}_{\alpha\eps} - \D_\beta {\Gamma^\delta}_{\alpha\gamma} - {\Gamma^\eps}_{\alpha\gamma}{\Gamma^\delta}_{\beta\eps}.\]
Now if $P \subset T_pX$ is spanned by a pair of orthogonal unit vectors $V^\alpha\D_\alpha$ and $W^\alpha\D_\alpha$, then $K(P) = R_{\alpha\beta\gamma\delta}V^\alpha W^\beta W^\gamma V^\delta$, and similarly for $\widetilde R$ and $\widetilde K$. 
Then
\[{R_{ijk}}^\ell =  {\widetilde{R}_{ijk}}^{\phantom{ijk}\ell} + {\Gamma^r}_{jk}{\Gamma^\ell}_{ir} - {\Gamma^r}_{ik}{\Gamma^\ell}_{jr} = {\widetilde R_{ijk}}^{\phantom{ijk}\ell} + (f^{-1})^2 (f')^2(-\delta^\ell_i g_{jk} + \delta^\ell_j g_{ik}),\]
\[{R_{rjk}}^r = \D_r{\Gamma^r}_{jk} - {\Gamma^m}_{rk}{\Gamma^r}_{jm} = - (f^{-1}f'g_{jk})' + (f^{-1}f')^2g_{jk} = -f^{-1}f''g_{jk}.\]
If  $\D_r \in P$ we take $V = \D_r$ and $W=W^j\D_j$ any unit vector in $T_pX$ orthogonal to $V$. Then 
\[K(P) = R_{rjkr}W^jW^k = -f^{-1}f''g_{jk}W^jW^k = -f^{-1}f''.\]
Meanwhile if $\D_r \perp P$ we may write $V=V^j\D_j$ and $W=W^j\D_j$. Then
\begin{align*}K(P) 
&= \left(f^2\tilde R_{ijk\ell}+ (f^{-1})^2 (f')^2(-g_{\ell i} g_{jk} + g_{\ell j} g_{ik})\right)V^iW^jW^kV^\ell.
\intertext{using the fact that $fV$ and $fW$ are orthogonal unit vectors for $\tilde g$, we see that}
K(P) &= f^{-2}\tilde K(P) - (f^{-1})^2 (f')^2.\end{align*}

\end{document}